\documentclass[11pt]{amsart}
\usepackage{amsfonts,amssymb,amsmath,amsthm}
\usepackage{bbm}
\usepackage{graphicx}
\usepackage{transparent}
\usepackage{caption}
\usepackage[labelfont=normalfont,labelformat=simple]{subcaption}

\usepackage{enumitem}
\usepackage{aliascnt} 
\usepackage{fancyhdr,url}
\usepackage[usenames,dvipsnames]{xcolor}
\usepackage{tikz}
\usepackage[all]{xy}
\usepackage[bookmarks=false]{hyperref}
\hypersetup{colorlinks=true,linkcolor=Cyan,citecolor=Cyan}

\makeatletter
\def\MyNewTheorem#1[#2]#3{%
  \newaliascnt{#1}{#2}
  \newtheorem{#1}[#1]{#3}
  \aliascntresetthe{#1}
  \expandafter\newcommand\csname #1autorefname\endcsname{#3}
}
\makeatother

\makeatletter
\newtheorem*{rep@theorem}{\rep@title}
\newcommand{\newreptheorem}[2]{%
\newenvironment{rep#1}[1]{%
 \def\rep@title{#2 \ref{##1}}%
 \begin{rep@theorem}}%
 {\end{rep@theorem}}}
\makeatother

\newtheorem{theorem}{Theorem}
\MyNewTheorem{cor}[theorem]{Corollary}
\MyNewTheorem{lemma}[theorem]{Lemma}
\newreptheorem{theorem}{Theorem}

\theoremstyle{definition}
\MyNewTheorem{definition}[theorem]{Definition}

\numberwithin{equation}{section}
\def\equationautorefname~#1\null{(#1)\null}
\def\itemautorefname~#1\null{#1\null}
\newcommand{\C}{{\ensuremath{\mathbb{C}}}}

\title{Diagrams for Relative Trisections}
\author{Nickolas A. Castro \and David T. Gay \and Juanita Pinz\'on-Caicedo}
\thanks{This work was supported by a grant from the Simons Foundation (\# 359873, David Gay)}
\date{}

\begin{document}
\begin{abstract}
We establish a correspondence between trisections of smooth, compact, oriented $4$--manifolds with connected boundary and diagrams describing these trisected $4$--manifolds. Such a diagram comes in the form of a compact, oriented surface with boundary together with three tuples of simple closed curves, with possibly fewer curves than the genus of the surface, satisfying a pairwise condition of being standard. This should be thought of as the $4$--dimensional analog of a sutured Heegaard diagram for a sutured $3$--manifold. We also give many foundational examples. 
\end{abstract}

\maketitle
\section{Introduction}

In~\cite{GKTrisections}, Gay and Kirby defined, and proved existence and uniqueness statements for, trisections of both closed $4$--manifolds and compact $4$--manifolds with connected boundary. In the latter, relative case, the trisections restrict to open book decompositions on the bounding $3$--manifolds. In the closed case, the same paper discusses trisection diagrams; these are diagrams involving curves on surfaces which uniquely determine closed, trisected $4$--manifolds up to diffeomorphism. The aim of this paper is to complete the story by defining relative trisection diagrams and showing that they uniquely determine trisected $4$--manifolds with connected boundary, as well as to present a series of fundamental examples.\\

Before recalling the background definitions in~\cite{GKTrisections}, we introduce some basic definitions and state the main result of the present article.

\begin{definition}
 Two $(n+1)$--tuples of the form $(\Sigma, \alpha^1, \ldots, \alpha^n)$, where each $\alpha^i$ is a collection $\alpha^i = \{\alpha^i_1, \ldots, \alpha^i_k\}$ of $k$ disjoint simple closed curves on the surface $\Sigma$, are {\em diffeomorphism and handle slide equivalent} if they are related by a diffeomorphism between the surfaces and a sequence of handle slides within each $\alpha^i$; i.e. one is only allowed to slide curves from $\alpha^i$ over other curves from $\alpha^i$, but not over curves from $\alpha^j$ when $j \neq i$.
\end{definition}

\begin{definition}
 A {\em $(g,k;p,b)$--trisection diagram} (where $2p+b-1\leq k\leq g+p+b-1$) is a $4$--tuple $(\Sigma,\alpha,\beta,\gamma)$, where $\Sigma$ is a surface of genus $g$ with $b$ boundary components and each of $\alpha$, $\beta$ and $\gamma$ is a collection of $g-p$ simple closed curves such that each triple $(\Sigma,\alpha,\beta)$, $(\Sigma,\beta,\gamma)$, and $(\Sigma,\gamma,\alpha)$ is diffeomorphism and handle slide equivalent to the triple $(\Sigma,\delta,\epsilon)$ shown in \autoref{F:StandardForPairs}.
\end{definition}

\begin{figure}[h]
\centering
\def\svgwidth{\textwidth}
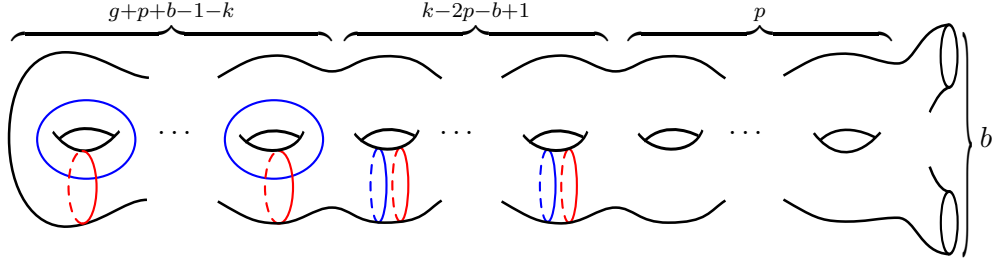
\caption{The standard model $(\Sigma,\delta,\epsilon)$.}\label{F:StandardForPairs}
\end{figure}

 The following theorem, the main result of this paper, references trisections of $4$--manifolds with boundary, but we defer the definition of this concept to a later section. If this is new to the reader, the main thing to know at the moment is that a trisection of a $4$--manifold $X$ is a decomposition into three codimension--$0$ submanifolds $X=X_1 \cup X_2 \cup X_3$, and that in the relative case a trisection induces an open book decomposition on $\partial X$.

 \begin{theorem} \label{T:DiagramsForTrisections}
  For every $(g,k;p,b)$--trisection diagram $(\Sigma,\alpha,\beta,\gamma)$ there is a unique (up to diffeomorphism) trisected $4$--manifold with connected boundary $X = X_1 \cup X_2 \cup X_3$ such that, with respect to a fixed identification $\Sigma \cong X_1 \cap X_2 \cap X_3$, the $\alpha$, $\beta$ and $\gamma$ curves, respectively, bound disks in $X_1 \cap X_2$, $X_2 \cap X_3$ and $X_3 \cap X_1$. In particular, the open book decomposition on $\partial X$ has $b$ binding components and pages of genus $p$. Furthermore, any trisected $4$--manifold with connected boundary is determined in this way by some relative trisection diagram, and any two relative trisection diagrams for the same $4$--manifold trisection are diffeomorphism and handle slide equivalent.
 \end{theorem}
 
 As a consequence, the monodromy of the open book decomposition on $\partial X$ is also completely determined by the diagram $(\Sigma,\alpha,\beta,\gamma)$. We now describe how to read off the monodromy from the diagram. 
 \begin{definition}
  Given a compact oriented surface $\Sigma$, consider a pair $(\alpha=(\alpha_1, \ldots, \alpha_k), a=(a_1, \ldots, a_l))$, where each $\alpha_i$ is a simple closed curve in $\Sigma$, each $a_j$ is a properly embedded arc in $\Sigma$, and $\{\alpha_1, \ldots, \alpha_k, a_1, \ldots, a_l\}$ are disjoint. We say that another such pair $(\alpha',a')$ is {\em handle slide equivalent} to $(\alpha,a)$ if $(\alpha',a')$ is obtained from $(\alpha,a)$ by a sequence of the following two operations: (1) Slide one simple closed curve in $\alpha$ over another simple closed curve in $\alpha$. (2) Slide one arc in $a$ over a simple closed curve in $\alpha$.
 \end{definition}
 Note that we do not allow ``arc slides'', in which arcs in $a$ slide over other arcs in $a$.

 We adopt the following notation: Given a surface $\Sigma$ and a collection of simple closed curves $\alpha$, $\Sigma_\alpha$ denotes the surface obtained by performing surgery along $\alpha$. This comes with an embedding $\phi_\alpha: \Sigma \setminus \alpha \to \Sigma_\alpha$, the image of which is the complement of a collection of pairs of points, one for each component of $\alpha$.

 \begin{theorem} \label{T:DiagramGivesMonodromy} 
 A relative trisection diagram $(\Sigma,\alpha,\beta,\gamma)$ encodes an open book decomposition on $\partial X$ with page given by $\Sigma_\alpha$, the surface resulting from $\Sigma$ by performing surgery along the $\alpha$ curves, and monodromy $\mu : \Sigma_\alpha \to \Sigma_\alpha$ determined by the following algorithm:
  \begin{enumerate}
   \item Choose an ordered collection of arcs $a$ on $\Sigma$, disjoint from $\alpha$ and such that its image $\phi_\alpha(a)$ in $\Sigma_\alpha$ cuts $\Sigma_\alpha$ into a disk.
   \item There exists a collection of arcs $a_1$ and simple closed curves $\beta'$ in $\Sigma$ such that $(\alpha,a_1)$ is handle slide equivalent to $(\alpha,a)$, $\beta'$ is handle slide equivalent to $\beta$, and $a_1$ and $\beta'$ are disjoint. (We claim that in this step we do not need to slide $\alpha$ curves over $\alpha$ curves, only $a$ arcs over $\alpha$ curves and $\beta$ curves over $\beta$ curves.) Choose such an $a_1$ and $\beta'$
   \item There exists a collection of arcs $a_2$ and simple closed curves $\gamma'$ in $\Sigma$ such that $(\beta',a_2)$ is handle slide equivalent to $(\beta',a_1)$, $\gamma'$ is handle slide equivalent to $\gamma$, and $a_2$ and $\gamma'$ are disjoint. (Again we claim that we do not now need to slide $\beta'$ curves over $\beta'$ curves.) Choose such an $a_2$ and $\gamma'$
   \item There exists a collection of arcs $a_3$ and simple closed curves $\alpha'$ in $\Sigma$ such that $(\gamma',a_3)$ is handle slide equivalent to $(\gamma',a_2)$, $\alpha'$ is handle slide equivalent to $\alpha$, and $a_3$ and $\alpha'$ are disjoint. (Now we do not need to slide $\gamma'$ curves over $\gamma'$ curves.) Choose such an $a_3$ and $\alpha'$.
   \item The pair $(\alpha',a_3)$ is handle slide equivalent to $(\alpha,a_*)$ for some collection of arcs $a_*$. Choose such an $a_*$. Note that now $a$ and $a_*$ are both disjoint from $\alpha$ and thus we can compare $\phi_\alpha(a)$ and $\phi_\alpha(a_*)$ in $\Sigma_\alpha$.
   \item The monodromy $\mu$ is the unique (up to isotopy) map such that $\mu(\phi_\alpha(a)) = \phi_\alpha(a_*)$, respecting the ordering of the collections of arcs.
  \end{enumerate}
 \end{theorem}
  
 Of course there are choices in the above algorithm each time we perform handleslides to arrange disjointness from the next system of curves, but part of the content of the theorem is that the resulting $\mu$ is independent of these choices.
 
 Note that this, together with the existence of trisections relative to given open books~\cite{GKTrisections}, gives us a purely $2$--dimensional result, namely that there is a way to encode mapping classes of surfaces with boundary via trisection diagrams (on higher genus surfaces).
 
 An alternative definition of a relative trisection diagram includes both the systems of curves $\alpha$, $\beta$ and $\gamma$ and the systems of arcs $a_1$, $a_2$, $a_3$; from such a definition it is easier to see that a diagram determines a trisected $4$--manifold. The nontriviality of both \autoref{T:DiagramsForTrisections} and \autoref{T:DiagramGivesMonodromy} is that one does not in fact need the arcs to uniquely determine the $4$--manifold and the open book on its boundary.

\section{Trisections of Closed Manifolds and their diagrams}
Let $Z_k = \natural^k (S^1 \times B^3)$ with $Y_k = \partial Z_k = \#^k (S^1 \times S^2)$. Given an integer $g \geq k$, let $Y_k = Y_{g,k}^- \cup Y_{g,k}^+$ be the standard genus $g$ Heegaard splitting of $Y_k$ obtained by stabilizing the standard genus $k$ Heegaard splitting $g-k$ times. 
\begin{definition} \label{D:Trisection}
A $(g,k)$--trisection of a closed, connected, oriented $4$--manifold $X$ is a decomposition of $X$ into three submanifolds $X = X_1 \cup X_2 \cup X_3$ satisfying the following properties:
\begin{enumerate}
 \item For each $i = 1,2,3$, there is a diffeomorphism $\phi_i : X_i \to Z_k$.
 \item For each $i = 1,2,3$, taking indices mod $3$, $\phi_i(X_i \cap
   X_{i+1}) = Y_{g,k}^-$ and $\phi_i(X_i \cap X_{i-1}) = Y_{g,k}^+$.
\end{enumerate}
\end{definition}

\begin{theorem}[Gay-Kirby~\cite{GKTrisections}]
 Every smooth closed oriented connected $4$--manifold has a trisection.
\end{theorem}

\begin{definition} \label{D:TrisectionDiagram}
 A {\em $(g,k)$--trisection diagram} is a tuple $(\Sigma, \alpha, \beta, \gamma)$ such that $\Sigma$ is a closed oriented surface of genus $g$ and each triple $(\Sigma,\alpha,\beta)$, $(\Sigma,\beta,\gamma)$ and $(\Sigma,\gamma,\alpha)$ is diffeomorphism and handle slide equivalent to the triple $(\Sigma,\delta,\epsilon)$ shown in \autoref{F:StandardForPairsClosed}.
\end{definition}

\begin{figure}[h]
\centering
\def\svgwidth{0.8\textwidth}
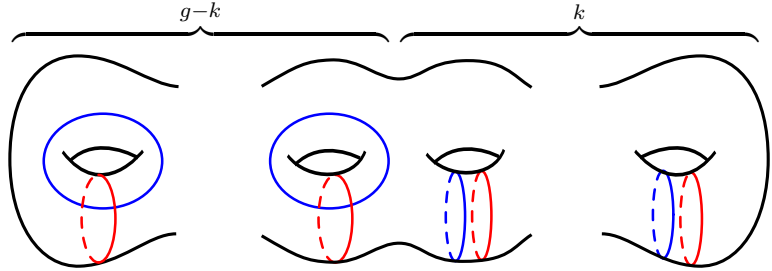
 \caption{The standard model $(\Sigma,\delta,\epsilon)$ in the closed case.}\label{F:StandardForPairsClosed} 
 \end{figure}

The following result is straightforward, and we present the proof here only to set the stage for the more subtle relative case.

\begin{theorem}[Gay-Kirby~\cite{GKTrisections}] \label{T:ClosedDiagrams} For every $(g,k)$--trisection diagram $(\Sigma,\alpha,\beta,\gamma)$ there is a unique (up to diffeomorphism) closed trisected $4$--manifold $X = X_1 \cup X_2 \cup X_3$ such that, with respect to a fixed identification $\Sigma \cong X_1 \cap X_2 \cap X_3$, the $\alpha$, $\beta$ and $\gamma$ curves, respectively, bound disks in $X_1 \cap X_2$, $X_2 \cap X_3$ and $X_3 \cap X_1$. Furthermore, any closed trisected $4$--manifold is determined in this way by some trisection diagram, and any two trisection diagrams for the same $4$--manifold trisection are diffeomorphism and handle slide equivalent.
\end{theorem}

\begin{proof} Note that the diagram in \autoref{F:StandardForPairsClosed} is a standard genus $g$ Heegaard diagram for $\#^k S^1 \times S^2 = Y_k$, describing the standard genus $g$ splitting $Y_k = Y_{g,k}^- \cup Y_{g,k}^+$. Fix an identification of $\Sigma$ with $Y_{g,k}^- \cap Y_{g,k}^+$ such that the $\delta$ curves bound disks in $Y_{g,k}^-$ and the $\epsilon$ curves bound disks in $Y_{g,k}^+$.
 
Given a trisected $4$--manifold $X=X_1 \cup X_2 \cup X_3$, let $\phi_i : X_i \to Z_k$, for $i=1,2,3$, be the diffeomorphisms from \autoref{D:Trisection}. The associated diagram is then $(X_1 \cap X_2 \cap X_3, \phi_1^{-1}(\delta), \phi_2^{-1}(\delta), \phi_3^{-1}(\delta))$. Equivalently one could replace any $\phi_i^{-1}(\delta)$ with $\phi_{i+1}^{-1}(\epsilon)$, or in fact any other cut system of $g$ curves bounding disks in $X_i \cap X_{i+1}$; the resulting diagrams would be handle slide equivalent~\cite{Johannson}.
 
Conversely, given a trisection diagram $(\Sigma,\alpha,\beta,\gamma)$, let $H_\alpha$, $H_\beta$ and $H_\gamma$, resp., be handlebodies bounded by $\Sigma$ determined by $\alpha$, $\beta$ and $\gamma$, resp. Then build $X$ by starting with $B^2 \times \Sigma$, attaching $I \times H_\alpha$, $I \times H_\beta$ and $I \times H_\gamma$ to $\partial B^2 \times \Sigma = S^1 \times F_g$ along successive arcs in $S^1$ crossed with $\Sigma$. This produces a $4$--manifold with three boundary components, but because each pair of systems of curves is a Heegaard diagram for $\#^k S^1 \times S^2$, each boundary component is diffeomorphic to $\#^k S^1 \times S^2$, and hence can be capped off uniquely with $\natural^k S^1 \times B^3$~\cite{LaudenbachPoenaru}.
\end{proof}

\section{Relative trisections}
Here we rephrase the definition of relative trisection from~\cite{GKTrisections}. Given integers $(g,k;p,b)$ with $g \geq p$ and $g+p+b-1 \geq k \geq 2p+b-1$, we begin as in the closed case with $Z_k = \natural^k S^1 \times B^3$ and $Y_k = \partial Z_k = \#^k S^1 \times S^2$, but in this case we describe a certain decomposition of $Y_k$ as $Y_k = Y_{g,k;p,b}^- \cup Y_{g,k;p,b}^0 \cup Y_{g,k;p,b}^+$ needed for the definition. This decomposition is illustrated in \autoref{F:RelModel} as a lower dimensional analogue.

 \begin{figure}[h]
 \centering
{\tiny
\def\svgwidth{\textwidth}
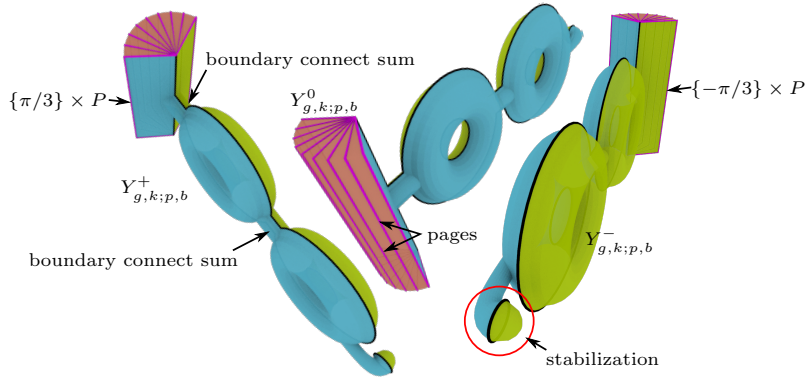 
}
 \caption{\label{F:RelModel} Several views of a lower dimensional analog of the standard model $Z_k$ for a sector of a relative trisection, with the decomposition of the boundary $Y_k = Y_{g,k;p,b}^- \cup Y_{g,k;p,b}^0 \cup Y_{g,k;p,b}^+$. The page $P$ is represented as a straight line segment, in purple. }
 \end{figure}

Let $D$ be a third of a unit 2-dimensional disk. Namely, use polar coordinates and set $D = \{(r,\theta) \mid r \in [0,1], \theta \in [-\pi/3,\pi/3]\}.$ Decompose $\partial D$ as $\partial D = \partial^- D \cup \partial^0 D \cup \partial^+ D$, where \begin{equation}\begin{aligned}&\partial^- D = \{r \in [0,1], \theta = -\pi/3\},\\ &\partial^0 D = \{ r = 1, \theta \in [-\pi/3,\pi/3]\}, \text{ and }\\ &\partial^+ D = \{r \in [0,1], \theta = \pi/3\}.\end{aligned}\label{cons::bdryD}\end{equation}

Now let $P$ be a compact surface of genus $p$ with $b$ boundary components and consider $U = D \times P$. Note that $U \cong \natural^{2p+b-1} S^1 \times B^3$ and that the decomposition \autoref{cons::bdryD} induces a decomposition of $\partial U$ as $$\partial U = \partial^- U \cup \partial^0 U \cup \partial^+ U,$$ where $\partial^\pm U = \partial^\pm D \times P$ and $\partial^0 U = (\partial^0 D \times P) \cup (D \times \partial P)$. Similarly, notice that if we regroup the sets involved in the decomposition of $\partial U$ into $\partial^- U \cup \partial^0 U$ and $\partial^+ U,$ we obtain the standard genus $2p+b-1$ Heegaard splitting of $\#^{2p+b-1} S^1 \times S^2$.

Next, decompose $\partial (S^1 \times B^3) = S^1 \times S^2$ as $\partial^- (S^1 \times B^3) \cup \partial^+ (S^1 \times B^3)$, where $\partial^\pm (S^1 \times B^3) = S^1 \times S^2_\pm$ and $S^2_\pm$ are the northern and southern hemispheres. This is the standard genus $1$ Heegaard splitting of $S^1 \times S^2$. For a positive integer $n$, let $V_n = \natural^{n} (S^1 \times B^3)$, with the boundary connect sums all occuring in neighborhoods of points in the Heegaard surface of each copy of $\partial (S^1 \times B^3)$, so that the induced decomposition $\partial V = \partial^- V \cup \partial^+ V$ is the standard genus $n$ Heegaard splitting of $\#^{n} (S^1 \times S^2)$. Now, given an integer $s \geq n$, let $\partial V_n = \partial^-_s V_n \cup \partial^+_s V_n$ be the result of stabilizing this Heegaard splitting exactly $s$ times. In what follows, to simplify notation, let $V=V_n$, where $n=k-2p-b+1$, and take $s$ to be $g-k+p+b-1$.

Finally, identify $Z_k$ with $U \natural V$, with the boundary connect sum connecting a neighborhood of a point in the interior of $\partial^- U \cap \partial^+ U$ with a neighborhood of a point in the Heegaard surface $\partial^-_s V \cap \partial^+_s V$. The induced decomposition of $Y_k = \partial Z_k$ is the advertised decomposition $Y_k = Y_{g,k;p,b}^- \cup Y_{g,k;p,b}^0 \cup Y_{g,k;p,b}^+$. To be more specific,\begin{equation}\begin{aligned}Y^\pm_{g,k;p,b} &= \partial^\pm U \natural \partial^\pm_s V\text{ and }\\Y^0_{g,k;p,b} &= \partial^0 U.\end{aligned}\label{cons::Y_k}\end{equation}

Before presenting the definition of a trisection relative to the boundary, we make a brief comment on the schematic representation of the stabilization in \autoref{F:RelModel}: The illustration shows a ``Heegaard splitting'' of a $2$--manifold, not a $3$--manifold, in which case ``stabilization'' corresponds to introducing a cancelling $0$--$1$--handle pair, or $1$--$2$--pair, depending on your perspective, and this is of course not as symmetric as stabilization in dimension $3$. In particular, the result is that one half of the splitting becomes disconnected while the other half remains connected. This is the best representation we can give when embedding the schematic in $\mathbb{R}^3$.

\begin{definition} \label{D:RelTrisection}
A {\em $(g,k;p,b)$--trisection} of a compact, connected, oriented $4$--manifold $X$ with connected boundary is a decomposition of $X$ into three submanifolds $X = X_1 \cup X_2 \cup X_3$ satisfying the following properties:
\begin{enumerate}
 \item For each $i = 1,2,3$, there is a diffeomorphism $\phi_i : X_i \to Z_k$.
 \item For each $i = 1,2,3$, taking indices mod $3$, $\phi_i(X_i \cap
   X_{i+1}) = Y_{g,k;p,b}^-$ and $\phi_i(X_i \cap X_{i-1}) = Y_{g,k;p,b}^+$, while $\phi_i(X_i \cap \partial X) = Y_{g,k;p,b}^0$.
\end{enumerate}
\end{definition}

\begin{lemma}
 A {\em $(g,k;p,b)$--trisection} of a compact, connected, oriented $4$--manifold $X$ with connected boundary induces an open book decomposition on $\partial X$ with pages of genus $p$ with $b$ boundary components.
\end{lemma}

\begin{proof}
 Each $X_i \cap \partial X$ is diffeomorphic to $Y_{g,k;p,b}^0$, which is diffeomorphic to $([-\pi/3,\pi/3] \times P) \cup (D \times \partial P)$. These three pieces fit together to form $\partial X$ precisely so that the three copies of $[-\pi/3 \times \pi/3] \times P$ form a bundle over $S^1$ with fiber $P$, and so that the three copies of $D \times \partial P$ form a $B^2 \times \partial P$, a disjoint union of solid tori that fill the boundary components of the bundle as neighborhoods of the binding components of an open book.
\end{proof}

\begin{theorem}[Gay-Kirby~\cite{GKTrisections}] Every smooth, compact, oriented, connected $4$--manifold with connected boundary, with a fixed open book decomposition on the boundary, has a trisection inducing the given open book.
\end{theorem}

\section{Relative trisections and sutured $3$--manifolds, and proofs of the main theorems}

In this section we make several observations about our model $(Z_k,Y_k)$. These observations will help us analyze the topology of the corresponding pieces of a relative trisection $X=X_1 \cup X_2 \cup X_3$ and will allow us to identify these spaces with more familiar ones.\\

\begin{enumerate}
 \item The intersection $Y_{g,k;p,b}^- \cap Y_{g,k;p,b}^+$, and hence the triple intersection $X_1 \cap X_2 \cap X_3$, is a surface of genus $g$ with $b$ boundary components. This is schematically illustrated in \autoref{F:RelModel} as a black $1$--manifold, see \autoref{F:SurfacesInModel}.
\begin{figure}[h!]
\centering
\def\svgwidth{0.9\textwidth}
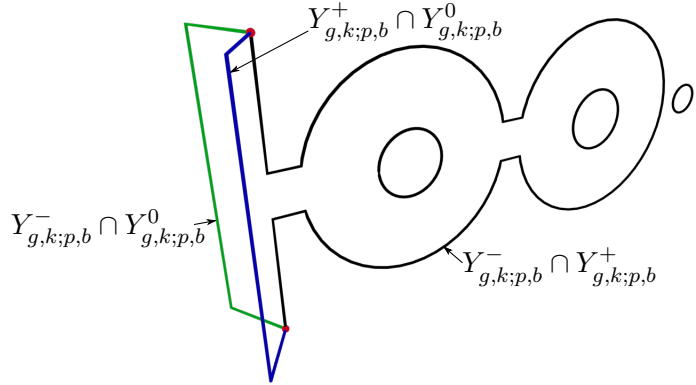
\caption{\label{F:SurfacesInModel} Three ``surfaces'' in the standard model, as represented in the lower dimensional schematic. Their common intersection, here shown as a red $S^0$, is really a disjoint union of $b$ copies of the circle $S^1$.}
\end{figure}
\item The intersection $Y_{g,k;p,b}^\pm \cap Y_{g,k;p,b}^0$, and hence $X_i \cap X_{i\mp 1} \cap \partial X$, is a surface of genus $p$ with $b$ boundary components, and so diffeomorphic to $P$. For $i=1,2,3$, these become three pages of the induced open book decomposition of $\partial X$. In \autoref{F:RelModel}, these appear as the two pink ends of the ``fan'' of pages; \autoref{F:SurfacesInModel} isolates the schematic representations of these two surfaces.
 \item The 3--dimensional triple intersection $Y_{g,k;p,b}^- \cap Y_{g,k;p,b}^0 \cap Y_{g,k;p,b}^+$, and hence the 4--dimensional intersection $X_1 \cap X_2 \cap X_3 \cap \partial X$, is a disjoint union of $b$ circles. These circles are precisely the components of $\partial P$, and as such, the binding of the induced open book. This appears schematically in \autoref{F:SurfacesInModel} as a red pair of points.
 \item Both $Y_{g,k;p,b}^-$ and $Y_{g,k;p,b}^+$, and hence $X_i \cap X_{i\pm 1}$, are $3$--dimensional relative compression bodies starting from a surface $\Sigma$ of genus $g$ with $b$ boundary components and compressing along $g-p$ disjoint simple closed curves to get to a surface $P$ of genus $p$ with $b$ boundary components. Here, by ``relative compression body'', we mean a cobordism with sides from a high genus surface at the bottom to a low genus surface at the top, each with the same number of boundary components, with a Morse function with critical points only of index $2$. The schematic representations of these two relative compression bodies are illustrated side by side in \autoref{F:CompressionBodies}.  
\item The union $Y_{g,k;p,b}^- \cup Y_{g,k;p,b}^+ = \overline{Y_k \setminus Y_{g,k;p,b}^0}$, and hence each $\overline{\partial X_i \setminus \partial X}$, is a {\em balanced sutured $3$--manifold}, with suture equal to a disjoint union of annuli described, in the explicit construction of $(Z_k,Y_k)$ described in \autoref{cons::Y_k}, as $\{r \in [0,1], \theta = \pm \pi/3\} \times \partial P$ with the first factor as in \autoref{cons::bdryD}. Thus each $\overline{\partial X_i \setminus \partial X}$ is a {\em balanced sutured $3$--manifold}, with suture $\Gamma$ equal to a regular neighborhood  in $\partial (\overline{\partial X_i \setminus \partial X})$ of the binding. The suture divides the boundary into two remaining pieces $P^-$ and $P^+$ which, in our case, are, respectively, $\{-\pi/3\} \times P$ and $\{\pi/3\} \times P$.  See \autoref{F:SuturedView}. Note that, in this paper, annular sutures of a sutured manifold are considered to be {\em parametrized annuli}, i.e. parametrized as $[-1,1] \times \partial P^-$.
 \item In fact the sutured manifold $Y_{g,k;p,b}^- \cup Y_{g,k;p,b}^+ = \overline{Y_k \setminus Y_{g,k;p,b}^0}$, and hence each $\overline{\partial X_i \setminus \partial X}=(X_i\cap X_{i-1})\cup(X_i\cap X_{i+1}),$ is diffeomorphic to $$([-1,1] \times P) \# (\#^{k-2p-b+1} S^1 \times S^2),$$ with suture $\Gamma = [-1,1] \times \partial P$ and boundary pieces $P^\pm = \{\pm 1\} \times P$. The decomposition as $Y_{g,k;p,b}^+ \cup Y_{g,k;p,b}^-$ is the connected sum of the decomposition of $[-1,1] \times P$ as $([-1,0] \times P) \cup ([0,1] \times P)$ with a $(g-k+b-1)$--times stabilized standard Heegaard splitting of $\#^{k-2p-b+1} S^1 \times S^2$. This gives a standard genus $g$ {\em sutured Heegaard splitting} of $\overline{\partial X_i \setminus \partial X}$.
 \item There is a diffeomorphism between the surface $\Sigma$ in \autoref{F:StandardForPairs} and $Y_{g,k;p,b}^+ \cap Y_{g,k;p,b}^-$ such that the $\delta$ curves in \autoref{F:StandardForPairs} bound disks in $Y_{g,k;p,b}^-$ and the $\epsilon$ curves in \autoref{F:StandardForPairs} bound disks in $Y_{g,k;p,b}^+$. Thus $(\Sigma,\delta,\epsilon)$ is a {\em sutured Heegaard diagram} for $Y_{g,k;p,b}^+ \cup Y_{g,k;p,b}^- = \overline{Y_k \setminus Y_{g,k;p,b}^0}$. (A sutured Heegaard diagram is a triple $(\Sigma, \delta,\epsilon)$ such that $\Sigma$ is a surface with boundary and each of $\delta$ and $\epsilon$ is a nonseparating collection of simple closed curves in $\Sigma$; such a diagram determines a sutured $3$--manifold, balanced if $|\delta| = |\epsilon|$.)
\end{enumerate}

 \begin{figure}[h]
  \begin{subfigure}{\textwidth}
 \centering
 \includegraphics[width=.8\textwidth]{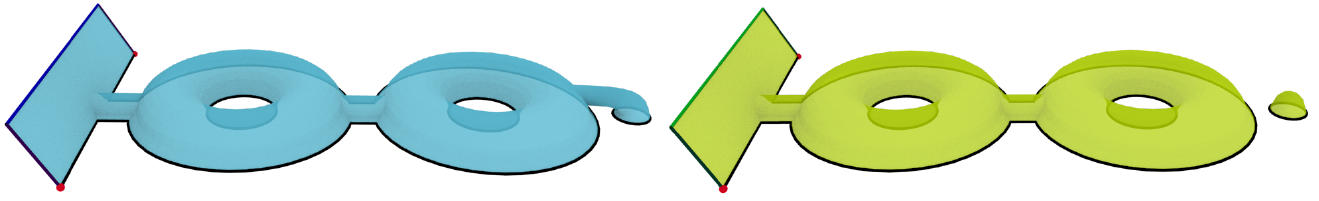}
 \caption{\label{F:CompressionBodies} The two relative compression bodies $Y_{g,k;p,b}^-$ and $Y_{g,k;p,b}^+$, each shown with the high genus ``surface'' $\Sigma$ on the bottom, the sides of the cobordism, slanted up and to the left, and the low genus ``page'' $P$ on the top.}
 \end{subfigure}
\begin{subfigure}{\textwidth}
 \centering
 \centering
\def\svgwidth{0.4\textwidth}
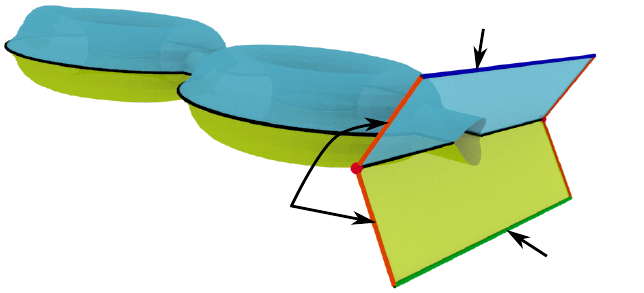
 \caption{\label{F:SuturedView} The two relative compression bodies fit together to form a sutured $3$--manifold, depicted here with ``sutures'' vertical and bent at a $2\pi/3$ angle along the core binding.}
 \end{subfigure}
 \caption{Diagrams concerning relative compression bodies.}
 \end{figure}

Notice that the decomposition of $Y_{g,k;p,b}$ into $Y_{g,k;p,b}^- \cup Y_{g,k;p,b}^0\cup Y_{g,k;p,b}^+$ can be modified into a decomposition with pieces $Y_{g,k;p,b}^- \cup Y_{g,k;p,b}^0$ and $Y_{g,k;p,b}^+$, by grouping together the first two pieces. This decomposition is the standard genus $k$ Heegaard splitting of $\#^{k} S^1\times S^2$ stabilized $g-k+p+b-1$ times. Notice also that $Y_{g,k;p,b}^0$ can be identified with a collar of the surface $P$ in $Y_{k,g+p+b-1}$. Thus, we can think of the space $Y_{g,k;p,b}^- \cup Y_{g,k;p,b}^+ = \overline{Y_k \setminus Y_{g,k;p,b}^0}$ as the complement of a surface with boundary in a Heegaard splitting and so it is only natural to expect arcs to be part of a notion of diagram for $Y_{g,k;p,b}^- \cup Y_{g,k;p,b}^+$. However, the last two observations indicate that it is possible to avoid the arcs. All this sets the stage for our main technical lemma.

\begin{lemma}
 Consider a diffeomorphism 
 $$\phi: ([-1,1] \times P) \# (\#^l S^1 \times S^2) \to ([-1,1] \times P) \# (\#^l S^1 \times S^2)$$ where the domain and range here are equipped with the sutured structure $\Gamma = [-1,1] \times \partial P$ and $P^\pm = \{\pm 1\} \times P$ discussed above. Suppose that $\phi|_{\Gamma \cup P^-} = \mathop{id}$. Then $\phi|_{P^+}$ is isotopic rel. boundary to the identity function $\mathop{id}: P^+ \to P^+$.
\end{lemma}

\begin{proof} To simplify notation, let $M=([-1,1] \times P) \# (\#^l S^1 \times S^2)$ and consider a properly embedded arc $a \subset P$; this gives rise to a simple closed curve $\gamma_a = (\{0\} \times a) \cup ([0,1] \times \partial a) \cup (\{1\} \times a) \subset \partial M$. Since $\phi|_{\Gamma \cup P^-} = \mathop{id}$, then $\phi(\gamma_a) = (\{0\} \times a) \cup ([0,1] \times \partial a) \cup (\{1\} \times a')$ for some other arc $a' \subset P$ with the same endpoints as $a$. Since $\gamma_a$ bounds a disk in $M$, so does $\phi(\gamma_a)$ and thus, in fact $\phi(\gamma_a)$ is homotopically trivial in $[-1,1] \times P$. Therefore the loop $\tau_a = a*(a')^{-1}$ obtained by concatenating $a$ and $(a')^{-1}$ is homotopically trivial in $P$. So $a$ and $a'$ are homotopic rel. endpoints, and thus by a result of Baer~\cite{Baer}, see~\cite[3.1]{Epstein}, $a$ and $a'$ are actually isotopic. Apply this to a collection of arcs cutting $P$ into a disk to conclude that $\phi|_{P^+}$ is isotopic rel. boundary to $\mathop{id}: P^+ \to P^+$.
\end{proof}

In what follows we use this lemma in the following form:

\begin{cor} \label{C:UniqueIXP}
 Consider the model sutured $3$--manifold $$(([-1,1] \times P) \# (\#^l S^1 \times S^2), \Gamma, P^-, P^+)$$ discussed above, and note that there is an ``identity'' map $\mathop{id} : P^- \to P^+$ defined by $\mathop{id}(-1,p) = (1,p)$. Given any sutured $3$--manifold $$(M,\Gamma_M,P^-_M,P^+_M)$$ diffeomorphic to $(([-1,1] \times P) \# (\#^l S^1 \times S^2), \Gamma, P^-, P^+)$ there is a unique (up to isotopy rel. boundary) diffeomorphism $\mathop{id}_M: P^-_M \to P^+_M$ such that, for any diffeomorphism $$\phi: (M,\Gamma_M,P^-_M,P^+_M) \to (([-1,1] \times P) \# (\#^l S^1 \times S^2), \Gamma, P^-, P^+),$$ we have $\mathop{id}_M = \phi^{-1} \circ \mathop{id} \circ \phi$.
\end{cor}

We are finally ready to prove the main results of this paper, namely \autoref{T:DiagramsForTrisections} and \autoref{T:DiagramGivesMonodromy}. We include the statements of both theorems again to make it easier for the reader to follow our proofs.

\begin{reptheorem}{T:DiagramsForTrisections}
 For every $(g,k;p,b)$--trisection diagram $(\Sigma,\alpha,\beta,\gamma)$ there is a unique (up to diffeomorphism) trisected $4$--manifold with connected boundary $X = X_1 \cup X_2 \cup X_3$ such that, with respect to a fixed identification $\Sigma \cong X_1 \cap X_2 \cap X_3$, the $\alpha$, $\beta$ and $\gamma$ curves, respectively, bound disks in $X_1 \cap X_2$, $X_2 \cap X_3$ and $X_3 \cap X_1$. In particular, the open book decomposition on $\partial X$ has $b$ binding components and pages of genus $p$. Furthermore, any trisected $4$--manifold with connected boundary is determined in this way by some relative trisection diagram, and any two relative trisection diagrams for the same $4$--manifold trisection are diffeomorphism and handle slide equivalent.
\end{reptheorem}

\begin{proof}[Proof of \autoref{T:DiagramsForTrisections}]
 We parallel as much as possible the proof of \autoref{T:ClosedDiagrams}.\\
 
 As mentioned above, the diagram $(\Sigma,\delta,\epsilon)$ in \autoref{F:StandardForPairs} is a {\em sutured Heegaard diagram} for $Y_{g,k;p,b}^+ \cup Y_{g,k;p,b}^- = \overline{Y_k \setminus Y_{g,k;p,b}^0}$. Fix an identification of $\Sigma$ with $Y_{g,k;p,b}^- \cap Y_{g,k;p,b}^+$ such that the $\delta$ curves bound disks in $Y_{g,k;p,b}^-$ and the $\epsilon$ curves bound disks in $Y_{g,k;p,b}^+$.
 
 Given a trisected $4$--manifold $X=X_1 \cup X_2 \cup X_3$, for $i=1,2,3$, let $\phi_i : X_i \to Z_k$ be the diffeomorphisms from \autoref{D:RelTrisection}. As before, the associated diagram is then $(X_1 \cap X_2 \cap X_3, \phi_1^{-1}(\delta), \phi_2^{-1}(\delta), \phi_3^{-1}(\delta))$. Equivalently one could replace any $\phi_i^{-1}(\delta)$ with $\phi_{i+1}^{-1}(\epsilon)$, or in fact any other complete non-separating system of curves bounding disks in $X_i \cap X_{i+1}$; the resulting diagrams would again be handle slide equivalent~\cite{Johannson, CassonGordon}.
 
 Conversely, given a relative $(g,k;p,b)$--trisection diagram $(\Sigma,\alpha,\beta,\gamma)$, let $C_\alpha$, $C_\beta$ and $C_\gamma$, resp., be relative compression bodies built by starting with $I \times \Sigma$ and attaching $3$--dimensional $2$--handles along $\alpha$, $\beta$ and $\gamma$, resp. The boundary of $C_\alpha$, for example, is naturally identified with $\Sigma \cup (I \times \partial \Sigma) \cup \Sigma_\alpha$, where $\Sigma_\alpha$ is the result of surgery applied to $\Sigma$ along $\alpha$. Let $P = \Sigma_\alpha \cong \Sigma_\beta \cong \Sigma_\gamma$. 
 
 Build $X$ by starting with $B^2 \times \Sigma$, attaching $I \times C_\alpha$, $I \times C_\beta$ and $I \times C_\gamma$ to $\partial B^2 \times \Sigma = S^1 \times \Sigma$ along the product of successive arcs in $S^1$ with $\Sigma$. This produces a $4$--manifold with boundary naturally divided into $B^2 \times \Sigma$, three copies of $(I \times P) \cup (I \times I \times \partial P)$ and three sutured $3$--manifolds diffeomorphic to $(([-1,1] \times P) \# (\#^l S^1 \times S^2), \Gamma, P^-, P^+)$. The three sutured manifolds are as advertised because each of $(\Sigma,\alpha,\beta)$, $(\Sigma,\beta,\gamma)$ and $(\Sigma,\gamma,\alpha)$ is handle slide and diffeomorphism equivalent to the standard sutured Heegaard diagram $(\Sigma,\delta,\epsilon)$ discussed above. This is illustrated in \autoref{F:AttachedHandlebodiesXI}. 
 \begin{figure}[h!]
 \centering
 \includegraphics[width=.4\textwidth]{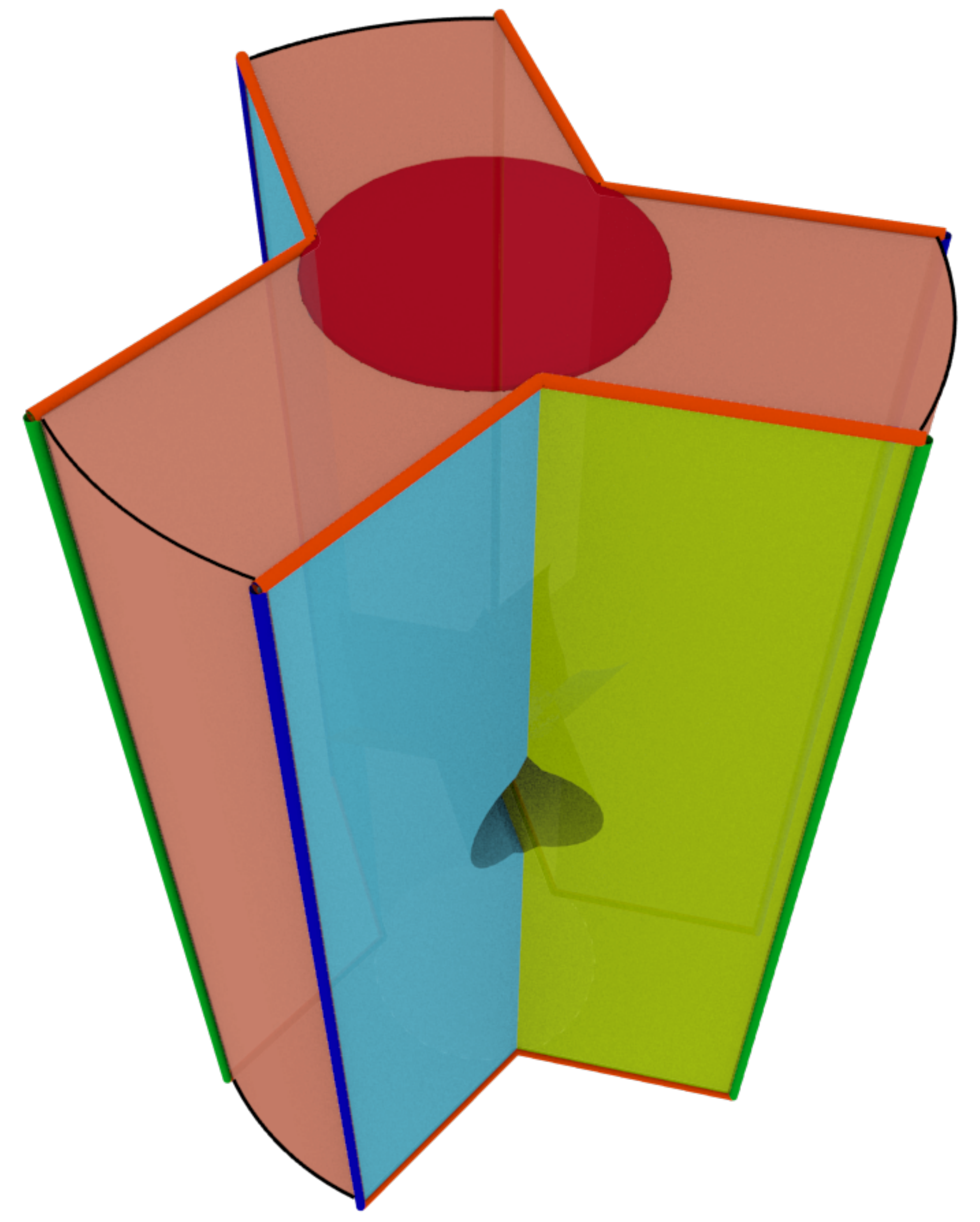}
 \caption{\label{F:AttachedHandlebodiesXI} $B^2 \times \Sigma$ with $I \times$ three relative compression bodies.}
 \end{figure}
 Using \autoref{C:UniqueIXP}, there is a unique way to glue $([-1,1] \times P) \cup (D \times \partial P)$, that is one third of an open book, to each of these sutured $3$--manifolds. Thickening the three pieces we have glued on to be $4$--dimensional, we get a $4$--manifold with four boundary components: one on the ``outside'', equal to an open book decomposition with page $P$, and three ``inside'' boundary components each diffeomorphic to $\#^k S^1 \times S^2$. This is illustrated in \autoref{F:UniqueIXP}, in which at the last stage we only see the outer boundary. Cap off each of the inside boundary components with $\natural^k S^1 \times B^3$ (uniquely, by~\cite{LaudenbachPoenaru}). The end result is our trisected $4$--manifold $X=X_1 \cup X_2 \cup X_3$. (Each $X_i$ is the union of a third of $B^2 \times \Sigma$, half of $I$ cross one relative compression body, half of $I$ cross the next relative compression body, the thickened copy of $([-1,1] \times P) \cup (D \times \partial P)$ glued in to this third, and the corresponding copy of $\natural^k S^1 \times B^3$.)
 \begin{figure}[h!]
 \centering
 \includegraphics[width=.9\textwidth]{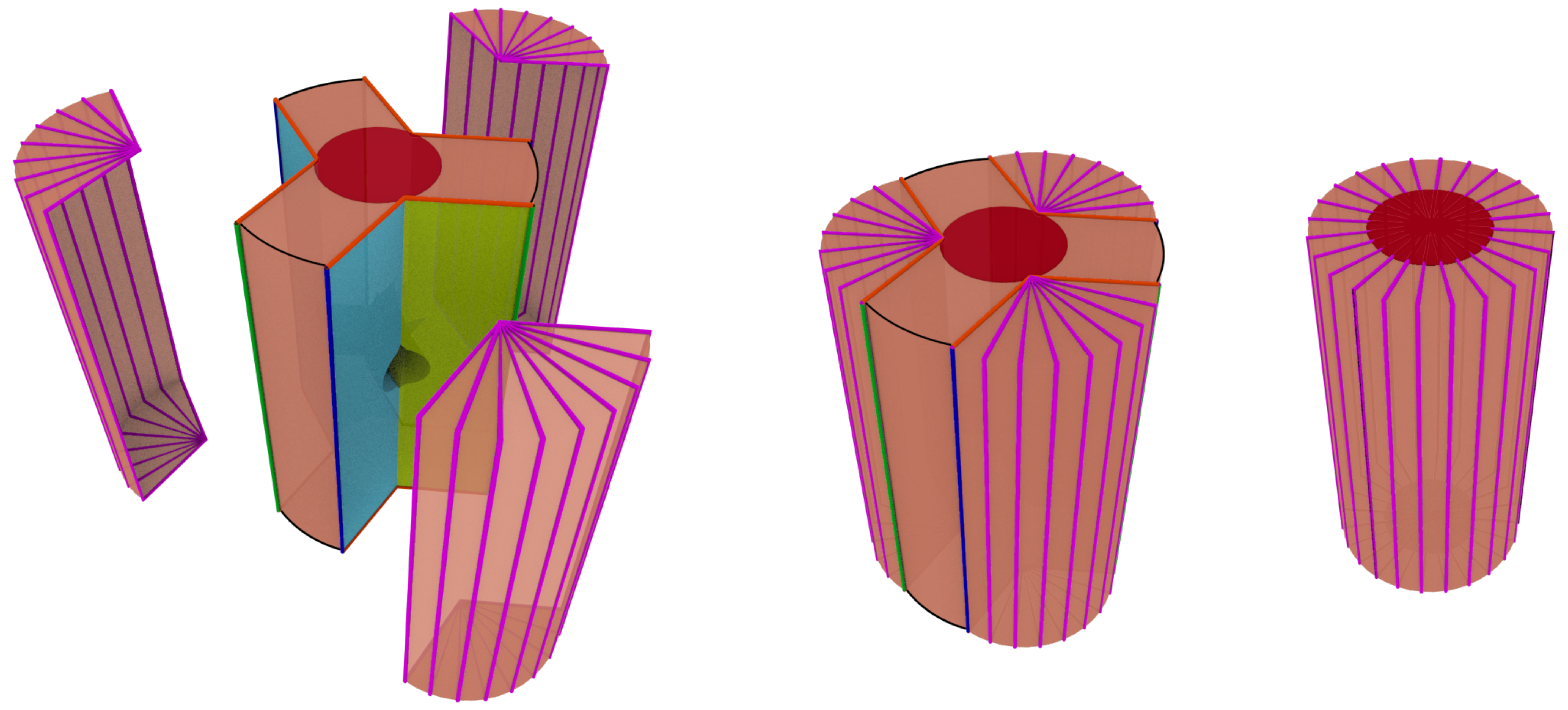}
 \caption{\label{F:UniqueIXP} Gluing on three groups of pages and closing up.}
 \end{figure}
 
\end{proof}

 \begin{reptheorem}{T:DiagramGivesMonodromy} 
 A relative trisection diagram $(\Sigma,\alpha,\beta,\gamma)$ encodes an open book decomposition on $\partial X$ with page given by $\Sigma_\alpha$, the surface resulting from $\Sigma$ by performing surgery along the $\alpha$ curves, and monodromy $\mu : \Sigma_\alpha \to \Sigma_\alpha$ determined by the following algorithm:
  \begin{enumerate}
   \item Choose an ordered collection of arcs $a$ on $\Sigma$, disjoint from $\alpha$ and such that its image $\phi_\alpha(a)$ in $\Sigma_\alpha$ cuts $\Sigma_\alpha$ into a disk.
   \item There exists a collection of arcs $a_1$ and simple closed curves $\beta'$ in $\Sigma$ such that $(\alpha,a_1)$ is handle slide equivalent to $(\alpha,a)$, $\beta'$ is handle slide equivalent to $\beta$, and $a_1$ and $\beta'$ are disjoint. (We claim that in this step we do not need to slide $\alpha$ curves over $\alpha$ curves, only $a$ arcs over $\alpha$ curves and $\beta$ curves over $\beta$ curves.) Choose such an $a_1$ and $\beta'$
   \item There exists a collection of arcs $a_2$ and simple closed curves $\gamma'$ in $\Sigma$ such that $(\beta',a_2)$ is handle slide equivalent to $(\beta',a_1)$, $\gamma'$ is handle slide equivalent to $\gamma$, and $a_2$ and $\gamma'$ are disjoint. (Again we claim that we do not now need to slide $\beta'$ curves over $\beta'$ curves.) Choose such an $a_2$ and $\gamma'$
   \item There exists a collection of arcs $a_3$ and simple closed curves $\alpha'$ in $\Sigma$ such that $(\gamma',a_3)$ is handle slide equivalent to $(\gamma',a_2)$, $\alpha'$ is handle slide equivalent to $\alpha$, and $a_3$ and $\alpha'$ are disjoint. (Now we do not need to slide $\gamma'$ curves over $\gamma'$ curves.) Choose such an $a_3$ and $\alpha'$.
   \item The pair $(\alpha',a_3)$ is handle slide equivalent to $(\alpha,a_*)$ for some collection of arcs $a_*$. Choose such an $a_*$. Note that now $a$ and $a_*$ are both disjoint from $\alpha$ and thus we can compare $\phi_\alpha(a)$ and $\phi_\alpha(a_*)$ in $\Sigma_\alpha$.
   \item The monodromy $\mu$ is the unique (up to isotopy) map such that $\mu(\phi_\alpha(a)) = \phi_\alpha(a_*)$, respecting the ordering of the collections of arcs.
  \end{enumerate}
 \end{reptheorem}

\begin{proof}[Proof of \autoref{T:DiagramGivesMonodromy}]

The fact that each of $(\Sigma,\alpha,\beta)$, $(\Sigma,\beta,\gamma)$ and $(\Sigma,\gamma,\alpha)$ is handle slide and diffeomorphism equivalent to the sutured Heegaard diagram $(\Sigma,\delta,\epsilon)$ in \autoref{F:StandardForPairs} tells us that we can in fact find the collections of arcs and sequences of slides advertised. Each time we find a collection of arcs which is disjoint from, for example, both $\beta$ and $\gamma$, this describes a diffeomorphism from $\Sigma_{\beta}$ to $\Sigma_{\gamma}$, which is the ``identity'' map coming from \autoref{C:UniqueIXP}. Thus we have the following steps:
\begin{enumerate}
 \item Note that $\phi_\alpha(a)$ is isotopic to $\phi_\alpha(a_1)$ in $\Sigma_\alpha$ because $a_1$ was produced from $a$ by sliding over $\alpha$ curves.
 \item Map $\Sigma_\alpha$ to $\Sigma_{\beta'}$ so as to send $\phi_\alpha(a_1) \subset \Sigma_\alpha$ to $\phi_{\beta'}(a_1) \subset \Sigma_{\beta'}$.
 \item Note that $\phi_{\beta'}(a_1)$ is isotopic to $\phi_{\beta'}(a_2)$ in $\Sigma_{\beta'}$ because $a_2$ was produced from $a_1$ by sliding over $\beta'$ curves.
 \item Map $\Sigma_{\beta'}$ to $\Sigma_{\gamma'}$ so as to send $\phi_{\beta'}(a_2) \subset \Sigma_{\beta'}$ to $\phi_{\gamma'}(a_2) \subset \Sigma_{\gamma'}$.
 \item Note that $\phi_{\gamma'}(a_2)$ is isotopic to $\phi_{\gamma'}(a_3)$ in $\Sigma_{\gamma'}$ because $a_3$ was produced from $a_2$ by sliding over $\gamma'$ curves.
 \item Map $\Sigma_{\gamma'}$ to $\Sigma_{\alpha'}$ so as to send $\phi_{\gamma'}(a_3) \subset \Sigma_{\gamma'}$ to $\phi_{\alpha'}(a_3) \subset \Sigma_{\alpha'}$.
 \item Map $\Sigma_{\alpha'}$ to $\Sigma_\alpha$ so as to send $\phi_{\alpha'}(a_3)$ to $\phi_\alpha(a_*)$.
\end{enumerate}
The fact that each of the maps in the above sequence of maps is independent of the choices is a restatement of \autoref{C:UniqueIXP}, and thus we see the monodromy expressed as a composition $\Sigma_\alpha \to \Sigma_{\beta'} \to \Sigma_{\gamma'} \to \Sigma_{\alpha'} \to \Sigma_\alpha$.

\end{proof}

\section{Examples}

\subsection{Disk Bundles over the 2--sphere $S^2$} \label{S:DiskBundles}
Consider $p:E_n\to S^2$ the oriented disk bundle over $S^2$ with Euler number $n$. Decompose $S^2$ as the union of three wedges $B_1,\; B_2,\; B_3$ that intersect pairwise in arcs joining the north and south pole and whose triple intersection consists precisely of the north and south poles as shown in \autoref{F:TrisectedS2}. Ideally, we would just lift this trisection of $S^2$ to get a trisection for $E_n$. However, although each $p^{-1}(B_i)$ is in fact a 4--dimensional 1--handlebody, the triple intersection of these pieces is not connected and so this naive decomposition of $E_n$ is not really a trisection. To fix this, for $i,j=1,2,3$ let $\varphi_i: B_i\times D^2\to p^{-1}(B_i)$ be a trivialization over $B_i$ and let $g_{ij}:B_i\cap B_j\to SO(2)$ be the transition function for $\varphi_i^{-1}\circ \varphi_j$. Next, parametrize each arc $B_i\cap B_{i+1}$ by $t\in [0,1]$ and use the cocycle condition to set 
\begin{equation}\label{D2xS2-transition}
g_{12},g_{23}: t\to 1,\text{ and } g_{31}:t\to e^{2\pi \mathbbm{i} nt}.
\end{equation} Here we are using the identification $e^{\mathbbm{i}\theta}=\left(\begin{array}{cc}\cos(\theta) & -\sin(\theta) \\ \sin(\theta) & \cos(\theta)\end{array}\right)$ and the notion of cocycle condition from \cite{davis-kirk}. In addition, choose sections $\sigma_i$ over $B_i$ ($i=1,2,3$) so that their images lie in the interior of the fiber and are disjoint from one another, let $\nu_i\cong B_i\times N_i$ be a tubular neighborhood of $\sigma_i\left(B_i\right)$ in $p^{-1}(B_i)$, and also assume that these tubular neighborhoods are interior in the fiber and pairwise disjoint. Finally, set $$X_i=\overline{p^{-1}\left(B_i\right)\setminus  \nu_i}\underset{\varphi_{i}\circ\varphi_{i+1}^{-1}}{\cup}  \nu_{i+1},$$ where the gluing is done via $\varphi_i\circ\varphi_{i+1}^{-1}:  \nu_{i+1}\cap p^{-1}(B_{i})\to  \nu_{i+1}\cap p^{-1}(B_{i})$. Notice that since $ \nu_i$ is a 2--handle, removing it from $p^{-1}(B_i)$ results in a space diffeomorphic to $S^1\times B^3$. In addition, since $\nu_{i+1}$ is attached along $ \nu_{i+1}\cap p^{-1}(B_{i})$ and this set is a 3--ball, attaching $\nu_{i+1}$ does not change the diffeomorphism type and thus $X_i$ is diffeomorphic to $S^1\times B^3$.\\

\begin{figure}[h]
\centering
\includegraphics[scale=.6]{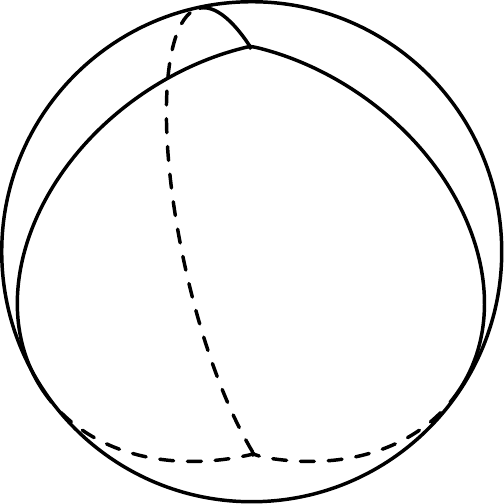}
\caption{\label{F:TrisectedS2} Decomposition of $S^2 = B_1 \cup B_2 \cup B_3$}
\end{figure}

For the $X_i$'s to define a trisection of $E_n$, we need to check that the intersections between them behave in the way stipulated in \autoref{D:RelTrisection}. With this in mind, consider first the pairwise intersection $X_{i-1}\cap X_i$ and notice that this intersection is such that %
\begin{equation}
\begin{split}
X_{i-1}\cap X_i= \overline{\left(p^{-1}\left(B_{i-1}\cap B_{i}\right)\setminus (\nu_{i-1}\cup \nu_{i})\right)} \\\underset{\varphi_{i-1}\circ \varphi_{i}^{-1}}{\cup} %
\partial_i\nu_i %
\underset{\varphi_{i}\circ \varphi_{i+1}^{-1}}{\cup} (\nu_{i+1}\cap p^{-1}(B_{i-1})).
\end{split}
\end{equation}%
Here $\left(p^{-1}\left(B_{i-1}\cap B_{i}\right)\setminus (\nu_{i-1}\cup \nu_{i})\right)$ is diffeomorphic to a 3-ball with two 2-handles removed, and $\nu_{i+1}\cap p^{-1}(B_{i-1})$ is a 1-handle. Moreover, the set $\partial_i\nu_i\cong B_i\times \partial N_i $, the boundary of $\nu_i$ as a subspace of $p^{-1}(B_i)$, is a solid torus attached to the 3-ball with two 2-handles removed along a cylinder in its boundary and thus is simply a thickening of one of the holes left by the 2-handles. We can then conclude that $X_{i-1}\cap X_i$ is diffeomorphic to a handlebody of genus 3. An extension of the previous argument then shows that the triple intersection is given by \begin{equation}\label{eq::D2xS2_triple}
\begin{split}
X_1\cap X_2 \cap X_3 =p^{-1}\left(B_{1}\cap B_{2}\cap B_3\right) \setminus  \left(\nu_1\cup  \nu_{2} \cup  \nu_3 \right)\\
\underset{\footnotesize\begin{array}{c}\varphi_{i}\circ \varphi_{i+1}^{-1}\\ i=1,2,3\end{array}}{\cup} 
\left[\bigcup_{i=1}^3 \partial_i\nu_i\cap p^{-1}(B_{i+1}) \right], 
\end{split}
\end{equation}
where $p^{-1}\left(B_{1}\cap B_{2}\cap B_3\right) \setminus  \left(\nu_1\cup  \nu_{2} \cup  \nu_3 \right)$ consists of the disjoint union of two 2--disks with three interior disks removed, and each $\partial_i\nu_i\cap p^{-1}(B_{i+1})$ is diffeomorphic to the cylinder $B_i\cap B_{i+1}\times \partial N_i$ and is glued to the first space in such a way that it joins internal boundary components of the two different disks. From this it follows that the triple intersection is a twice punctured genus two surface. The last intersections to consider are those that involve the boundary, namely, $X_i\cap E_n$ and $X_{i-1}\cap X_i\cap \partial E_n$. In this case we have $$X_i\cap \partial E_n=\partial p^{-1}(B_i)\setminus p^{-1}(\partial B_i)\cong B_i\times\partial D^2,$$ and $$X_{i-1}\cap X_i\cap \partial E_n=\partial p^{-1}(B_{i-1}\cap B_i)\setminus p^{-1}(\partial B_{i-1}\cap \partial B_i)\cong B_{i-1}\cap B_i\times \partial D^2.$$ From this we see that $X_i\cap E_n$ is diffeomorphic to $I\times X_{i-1}\cap X_i\cap \partial E_n$ with the space $\partial I\times X_{i-1}\cap X_i\cap \partial E_n$ identified, or, using the terminology of \autoref{cons::bdryD} that $X_i\cap E_n$ is diffeomorphic to $\partial ^0 D\times (X_{i-1}\cap X_i\cap \partial E_n) \cup D\times \partial (X_1\cap X_2\cap X_3)$. \\

In sum, the previous paragraphs describe a $(2,1;0,2)$ relative trisection of $E_n$ whose relative trisection diagram $(\Sigma,\alpha,\beta,\gamma)$ has yet to be exhibited. To this end, notice that by \autoref{eq::D2xS2_triple}, $\Sigma$ is a surface decomposed as the union of two copies of a three times punctured disk with three cylinders joining the punctures of the two disks. To finish the description of the diagram, it is enough to find three sets of curves in $F=X_1\cap X_2\cap X_3$ that bound disks in the double intersections $X_{i-1}\cap X_{i}$, and draw their images in $\Sigma$. For example, in $X_{3}\cap X_1$, the 1-handle $\nu_2\cap p^{-1}(B_3)$ has the cylinder $\partial_2\nu_2\cap p^{-1}(B_{3})$ as its boundary and so the central circle in the latter is one of the curves in the collection $\gamma$. A similar argument applied to the other two pairwise intersections shows that the central circle in $\partial_3\nu_3\cap p^{-1}(B_{1})$ is a curve in $\alpha$ and $\partial_1\nu_1\cap p^{-1}(B_{2})$ is a curve in $\beta$. Next, consider the disk $\mathcal{D}$ in $X_3\cap X_1$ constructed as the union of a disk in $p^{-1}(B_3\cap B_1)\setminus \nu_3\cup\nu_1$ that lies between the holes left by $\nu_3,\,\nu_1$ with a meridional disk in $\partial_3\nu_3$. Then, the curve $\partial \mathcal{D}$ can be realized as the union of: 
\begin{enumerate}[label=(\roman*)]
\item\label{arc-vert} a properly embedded arc in $ \partial_{3}\nu_{3}\cap p^{-1}(B_{1})$ with one endpoint in each boundary component,
\item\label{arc-twist} a properly embedded arc in $ \partial_{1}\nu_{1}\cap p^{-1}(B_{2})$ with one endpoint in each boundary component, and
\item two horizontal arcs that lie in different components of the disjoint union of disks $\left(p^{-1}\setminus  \nu_1\cup  \nu_{2} \cup  \nu_3 \right)\left(B_{1}\cap B_{2}\cap B_3\right)$. 
\end{enumerate}
This curve $\partial \mathcal{D}$ is the second curve in the collection $\gamma$ and to draw it in $\Sigma$ we have to proceed with caution since by assumption the gluing map $\varphi_3\circ\varphi_1^{-1}$ depends on $n$. Indeed, using \autoref{D2xS2-transition} we see that the two disks that make up $\mathcal{D}$ align only if the second one is twisted. Thus, the arc described in \autoref{arc-twist}  appears in $ \partial_{1}\nu_{1}\cap p^{-1}(B_{2})$ as an arc with $n$-twists. Lastly, to get the remaining curves in $\alpha$ and $\beta$, we proceed in a similar manner noticing that in these cases the gluing maps are trivial and thus the analogous arcs to the one from \autoref{arc-twist} are not twisted. This shows that the trisection diagram corresponding to the decomposition $E_n=X_1\cup X_2\cup X_3$ can be obtained from the one shown in \autoref{fig::D2xS2} by replacing the single left handed twist on the green curve appearing in the right, with $n$ full twists around the cylinder. 
\begin{figure}[h]
\includegraphics[scale=.8]{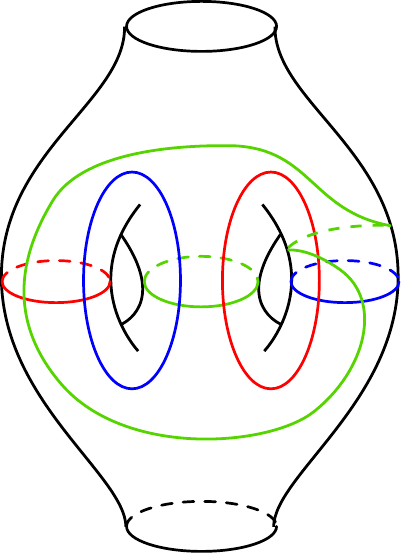}
\caption{A $(2,1;0,2)$ relative trisection diagram for the disk bundle over $S^2$ corresponding to the integer -1. The monodromy of the open book in the boundary is a left handed twist}\label{fig::D2xS2}
\end{figure}

\subsection{Local modifications of diagrams, Lefschetz fibrations and Hopf plumbings}

Throughout this section, suppose that we are given a relative trisection diagram $(\Sigma,\alpha,\beta,\gamma)$ for a trisected $4$--manifold $X=X_1 \cup X_2 \cup X_3$, with induced open book on $\partial X$ with page $P = \Sigma_\alpha$ and monodromy $\mu: P \to P$.

\begin{lemma} \label{L:1Handle}
 Let $\Sigma' \supset \Sigma$ be the result of attaching a $2$--dimensional $1$--handle to $\Sigma$ along some $S^0 \subset \partial \Sigma$. Then the tuple $(\Sigma',\alpha,\beta,\gamma)$ is a relative trisection diagram for a trisected $4$--manifold $X' = X_1' \cup X_2' \cup X_3'$ such that $X'$ is the result of attaching a $4$--dimensional $1$--handle $H$ to $X$ along the same $S^0 \subset \partial \Sigma$, seeing $\partial \Sigma \subset \partial X$ as the binding of the open book on $\partial X$. Furthermore, $H = H_1 \cup H_2 \cup H_3$, where each $H_i$ is a $4$--dimensional $1$--handle attached to $X_i$ to form $X_i'$. The open book on $\partial X'$ has page $P' = P \cup h$, the result of attaching the $2$--dimensional $1$--handle $h$ to $P$, and monodromy $\mu'$ equal to $\mu$ extended by the identity across $h$.
\end{lemma}
\begin{proof}
 Let $h$ be the $2$--dimensional $1$--handle attached to $\Sigma$ to form $\Sigma'$. In the construction of $X$ and $X'$, we see that $X$ is naturally a subset of $X'$ and that $X' \setminus X$ is precisely a $1$--handle $H = B^2 \times h$. Splitting $B^2$ into three thirds $B^2 = D_1 \cup D_2 \cup D_3$ gives the three $1$--handles $H_i = D_i \times h$.
\end{proof}

\begin{lemma} \label{L:2Handle}
 Consider a simple closed curve $C \subset \Sigma$ disjoint from $\alpha$ and transverse to $\beta$ and $\gamma$. Let $(\Sigma^\pm,\alpha^\pm,\beta^\pm,\gamma^\pm)$ be the result of removing a cylinder neighborhood of $C$, together with the $\beta$ and $\gamma$ arcs running across this neighborhood, and replacing it with a twice-punctured torus as in \autoref{F:LocalVanCycle} with $\beta$ and $\gamma$ arcs as drawn, and with one new $\alpha$, $\beta$ and $\gamma$ curve as drawn. Then $(\Sigma^\pm,\alpha^\pm,\beta^\pm,\gamma^\pm)$ is a relative trisection diagram for a trisected $4$--manifold $X' = X_1' \cup X_2' \cup X_3'$ such that $X'$ is the result of attaching a $2$--handle to $X$ along $C \subset P$ with framing $\mp 1$ relative to $P$, and such that the open book on $\partial X'$ has page $P$ with monodromy $\tau_C^{\pm 1} \circ \mu$, where $\tau_C$ is a right-handed Dehn twist about $C$.
\end{lemma}
\begin{figure}[h]
\def\svgwidth{0.65\textwidth}
			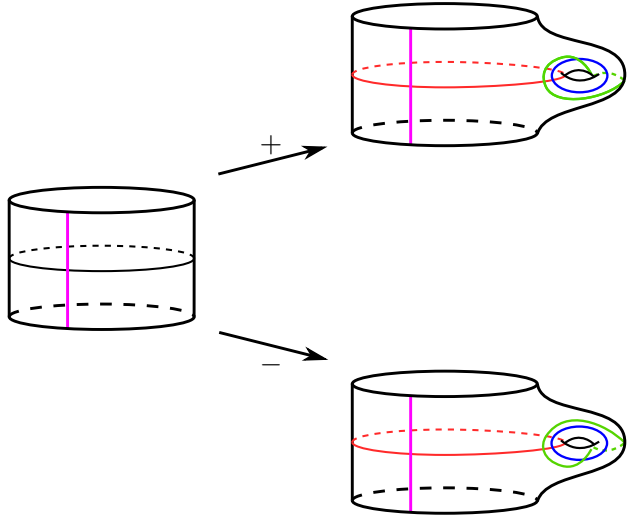
\caption{Local modification of $(\Sigma,\alpha,\beta,\gamma)$ near a curve $C$ disjoint from $\alpha$ and transverse to $\beta$ and $\gamma$. The pink transverse arc represents a collection of parallel $\beta$ and $\gamma$ arcs.}\label{F:LocalVanCycle}
\end{figure}
\begin{proof}
 Since $(\Sigma,\alpha,\beta,\gamma)$ is a trisection diagram, we know that there is an arc $A$ connecting $C$ to $\partial X$ avoiding $\alpha$ and transverse to $\beta$ and $\gamma$; we draw a neighborhood of $C \cup A$ as on the left in \autoref{F:LocalVanV2}. In this picture there are two groups of $\beta$ and $\gamma$ arcs: those transverse to $C$ and those transverse to $A$. The modification drawn in \autoref{F:LocalVanCycle} is then redrawn in \autoref{F:LocalVanV2} so that we see the new genus in $\Sigma'$ as arising from $\Sigma$ by attaching two $2$--dimensional $1$--handles $h_1$ and $h_2$. The $\beta$ and $\gamma$ arcs that were transverse to $A$ avoid the new $\alpha$, $\beta$ and $\gamma$ curve by running parallel to $\partial \Sigma'$. Note that we can slide these boundary-parallel $\beta$ and $\gamma$ arcs over the new $\beta$ or, respectively, $\gamma$ curve to get \autoref{F:LocalVanV3}. (Each $\beta$, resp. $\gamma$, arc slides twice over the $\beta$, resp. $\gamma$, curve.) Thus we can take \autoref{F:LocalVanV3} to be the modification of the trisection diagram which we work with; i.e. $(\Sigma^\pm,\alpha^\pm,\beta^\pm,\gamma^\pm)$ is obtained from $(\Sigma,\alpha,\beta,\gamma)$ by replacing the figure on the left in \autoref{F:LocalVanV2} with \autoref{F:LocalVanV3}.
\begin{figure}[h]
\def\svgwidth{0.75\textwidth}
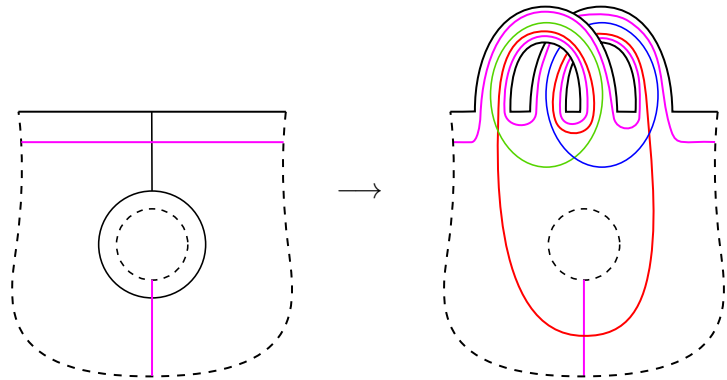
\caption{A different perspective of the local modification of $(\Sigma,\alpha,\beta,\gamma)$, taking into account an arc $A$ connecting $C$ to $\partial \Sigma$. Again, the pink arcs represent collections of parallel $\beta$ and $\gamma$ arcs; now one collection of such arcs is transverse to the closed curve $C$ and one collection is transverse to the arc $A$.}\label{F:LocalVanV2}
\end{figure}
\begin{figure}[h]
\def\svgwidth{0.3\textwidth}
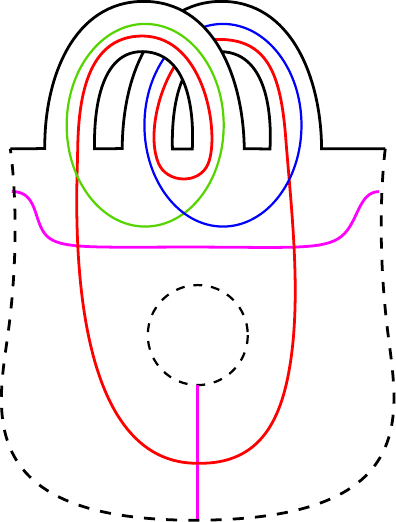
\caption{After some handle slides.}\label{F:LocalVanV3}
\end{figure}

Now, recalling the construction of $X$ from the diagram $(\Sigma,\alpha,\beta,\gamma)$ and of $X'$ from $(\Sigma^\pm,\alpha^\pm,\beta^\pm,\gamma^\pm)$, we see that $X'$ is naturally built by adding two $4$--dimensional $1$--handles to $X$ (as in \autoref{L:1Handle}) followed by three $4$--dimensional $2$--handles, one along the new $\alpha$ curve in $\Sigma'_\alpha$, one along the new $\beta$ curve in $\Sigma'_\beta$ and one along the new $\gamma$ curve in $\Sigma'_\gamma$, with $0$--framings relative to the pages in which they sit. The $\beta$ and $\gamma$ $2$--handles each, topologically, cancel one of the new $1$--handles, and when this cancellation is performed, we see that the $\alpha$ curve now sits in $\Sigma_\alpha$ with framing equal to $\pm 1$ with respect to $\Sigma_\alpha$.

 \autoref{F:LocalVanMon} shows a local implementation of the algorithm from \autoref{T:DiagramGivesMonodromy} to show the effect of the new monodromy on a single arc transverse to $C$, thus completing the proof of the lemma.

\begin{figure}[h]
\centering
\xymatrix@M=1em{
\Sigma^+\includegraphics[width=.19\textwidth]{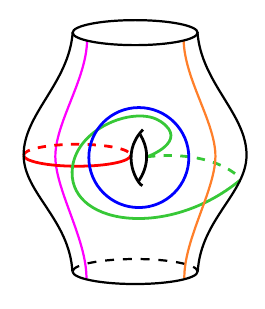}\ar@<1.25cm>[r]\ar[d]&
\includegraphics[width=.195\textwidth]{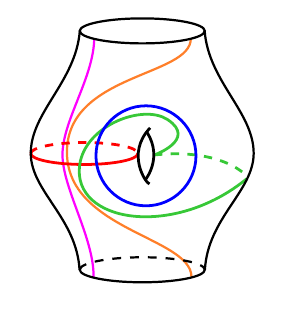}\ar@<1.25cm>[r]&
\includegraphics[width=.19\textwidth]{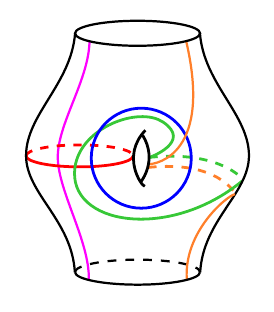}\ar[d]\\
P\includegraphics[width=.125\textwidth]{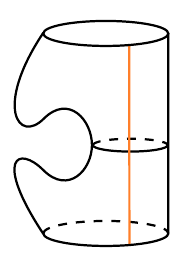}\ar@<1cm>[rr]^{\mu}&&
\includegraphics[width=.125\textwidth]{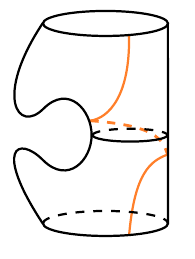}
}
\caption{Local effect on the monodromy.}\label{F:LocalVanMon}
\end{figure}
 
\end{proof}

Note that the roles of $\alpha$, $\beta$ and $\gamma$ in \autoref{L:2Handle} can obviously be cyclically permuted; in some of the following applications, $\gamma$ will play the role that $\alpha$ plays here.

We have two immediate corollaries. The first describes a stabilization operation on trisection diagrams corresponding to Hopf plumbing on the bounding open book decomposition, and is the diagrammatic version of the construction described in section 3.3 of \cite{NickThesis}.
\begin{cor} Suppose that $X$ has a trisection $T$ with induced open book decomposition $D$ on $\partial X$, and that $D^+$ (resp. $D^-$) is an open book decomposition of $\partial X$ obtained from $D$ by plumbing a left-handed (resp. right-handed) Hopf band along a properly embedded arc $A$ in a page $P$ of the open book $D$. If $T$ is described by the relative trisection diagram $(\Sigma,\alpha,\beta,\gamma)$ such that $P$ is identified with $\Sigma_\alpha$, consider the new diagram $(\Sigma^{'\pm},\alpha^{\pm},\beta^{\pm},\gamma^{\pm})$ obtained by first attaching a $2$--dimensional $1$--handle to $\Sigma$ at the end points of $A$, as in \autoref{L:1Handle}, producing $(\Sigma',\alpha,\beta,\gamma)$ and then modifying this as in \autoref{L:2Handle} in a neighborhood of the curve $C$ obtained by attaching the core of the $1$--handle to the arc $A$. Then $(\Sigma^{'\pm},\alpha^{\pm},\beta^{\pm},\gamma^{\pm})$ is again a trisection of $X$ inducing the open book decomposition $D^\pm$ on $\partial X$. \end{cor}

We leave the proof of this corollary to the reader.\\

For the next corollary, let $P$ be a smooth orientable surface with boundary and for $c$ a curve embedded in $P$, denote by $\tau_c$ the right handed twist of $P$ along $c$. Given a 3--manifold $Y$ with open book decomposition given by $(P,\mu)$ with $\mu$ factored as $\mu=\tau_{c_n}^{\epsilon_n}\circ\ldots\circ \tau_{c_1}^{\epsilon_1}$ with $\epsilon_i\in\{-1,1\}$, and $c_i$ a curve in $P$, $i=1,\ldots,n$ it is well known that $Y$ is the boundary of a 4--manifold $X$ admitting an achiral Lefschetz fibration over $D^2$ with vanishing cyles $c_1,\ldots,c_n$. Moreover \cite{kas}, $X$ admits a handle decomposition diffeomorphic to the result of attaching $n$ 2--handles $h^2_1,\ldots,h^2_n,$ to $D^2\times P$ along the circles $\{1\}\times c_i$ with framing given by the surface framing minus $\epsilon_i$.

\begin{cor} \label{C:LFtoTrisection}
 Let $\pi: X\to D^2$ be an achiral Lefschetz fibration with regular fiber a surface $P$ of genus $p$ and $b$ boundary components, and with $n$ vanishing cyles. The manifold $X$ admits a $(p+n,2p+b-1;p,b)$ trisection. 
\end{cor}

\begin{proof}
Build $X$ and its trisection beginning with the standard $(0,0;0,1)$ trisection of $B^4$ and attaching $1$--handles as in \autoref{L:1Handle} to produce $P \times D^2$ with a trisection inducing the standard open book on $P \times S^1$ with page $P$ and identity monodromy. At this stage the central surface $\Sigma^0$ is $P$, and there are no $\alpha$, $\beta$ or $\gamma$ curves. Attach a $2$--handle along $c_1$ as in \autoref{L:2Handle} to get a new $(\Sigma^1,\alpha^1,\beta^1,\gamma^1)$, such that each of $\alpha^1$, $\beta^1$ and $\gamma^1$ consists of a single curve, and $P$ is identified with $\Sigma^1_{\alpha^1}$. Now, as $i$ goes from $2$ to $n$ repeat the following process: Pull $c_i$ back from $P$ to $\Sigma^{i-1}$, using the fact that $P$ is identified with $\Sigma^{i-1}_{\alpha^{i-1}}$, and then apply \autoref{L:2Handle} to $c_i \subset \Sigma^{i-1}$ to produce $(\Sigma^i,\alpha^i,\beta^i,\gamma^i)$, with $P$ again identified with $\Sigma^i_{\alpha^i}$.
\end{proof}

The subtlety in implementing the method of proof above in a particular example arises when the vanishing cycles intersect. The images in \autoref{F:Lantern} illustrate a slightly nontrivial example, in which the vanishing cycles correspond to one side of the lantern relation in the mapping class group of a genus $0$ surface with four boundary components. The end result is a relative trisection diagram for a well known rational homology $4$--ball with boundary $L(4,1)$; see~\cite{EndoGurtas, FinSternRationalBD}. Note that from \autoref{F:Lantern3} to \autoref{F:Lantern4} we need to isotope the third vanishing cycle so as to be disjoint from a red $\alpha$ curve before proceeding to \autoref{F:Lantern5}. This corresponds to adjusting our drawing so that the third vanishing cycle does in fact live in the page obtained by surgering the central surface along the $\alpha$ curves. 

\begin{figure}[h]
\centering
\begin{subfigure}[b]{0.4\textwidth}\centering
\includegraphics[width=.75\textwidth]{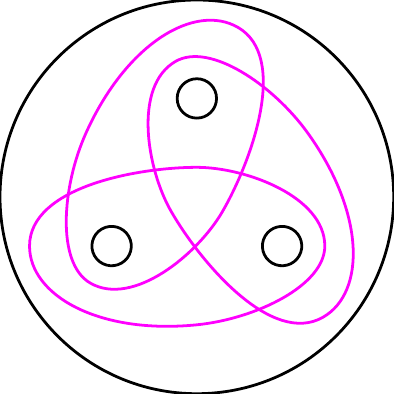}
\caption{Three vanishing cycles on a genus $0$ surface with $4$ boundary components.}\label{F:Lantern1}
\end{subfigure}

\vspace*{11pt}
\begin{tabular}{cc}
\begin{subfigure}[b]{0.4\textwidth}
\centering
\includegraphics[width=.75\textwidth]{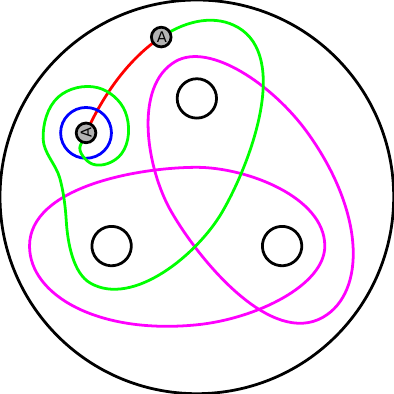}
\caption{One vanishing cycle turned into $\alpha$, $\beta$ and $\gamma$ curves, genus now equal to $1$.}\label{F:Lantern2}
\end{subfigure}&
\begin{subfigure}[b]{0.4\textwidth}\centering
\includegraphics[width=.75\textwidth]{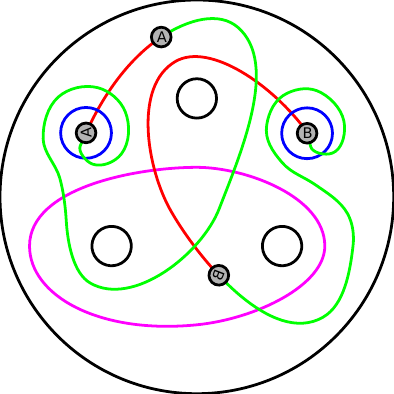}
\caption{Two vanishing cycles done, genus equals $2$; note that $C_3$ now intersects $\alpha$ curves.}\label{F:Lantern3}
\end{subfigure}\\[11pt]
\begin{subfigure}[t]{0.4\textwidth}\centering
\includegraphics[width=.75\textwidth]{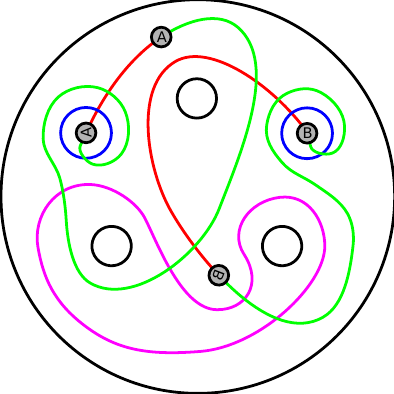}
\caption{$C_3$ isotoped to intersect only $\gamma$ curves.}\label{F:Lantern4}
\end{subfigure}&
\begin{subfigure}[t]{0.4\textwidth}\centering
\includegraphics[width=.75\textwidth]{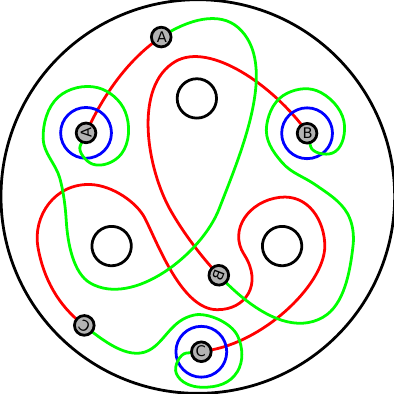}
\caption{A rational homology $B^4$.}\label{F:Lantern5}
\end{subfigure}
\end{tabular}
\caption{A relative trisection diagram for a rational homology $4$--ball with boundary $L(4,1)$.}\label{F:Lantern}
\end{figure}

\subsection{Plumbings}
In this section, we explain how to combine the method to obtain a diagram for achiral Lefschetz fibrations with well-known facts about plumbings of disk bundles over surfaces to describe trisection diagrams for plumbings of disk bundles. Notice however that for a single disk bundle of large Euler class, this method gives a much higher genus trisection than the method in \autoref{S:DiskBundles}.

\begin{definition}\label{def::plumb} A plumbing graph is a finite connected graph $\Gamma$ whose vertices and edges are assigned weights as follows:
\begin{itemize}
\item each vertex $v$ of $\Gamma$ carries two integer weights $e_v$, and $g_v$, with $g_v\geq 0$,
\item each edge of $\Gamma$ is assigned a sign +1 or -1. 
\end{itemize}
\end{definition}
To simplify notation, denote by $V(\Gamma)$ the set of vertices, $E(\Gamma)$ the set of edges, and $Q(\Gamma)$ the incidence matrix of $\Gamma$, that is, the matrix whose $q_{vw}$ entry is given by the signed count of edges joining the vertices $v$ and $w$ if $v\neq w$, and $q_{vv}=e_v$. In addition, for every vertex $v$ let $s_v=\sum_{w\in V(\Gamma)}q_{vw}$. Then, if $d_v$ is the degree of $v$, or in other words the weighted sum of edges that intersect $v$, we have $s_v=e_v+d_v$.

\begin{definition} Given a plumbing graph $\Gamma$, its modified plumbing graph is the connected graph $\Gamma^*$ that results from adding loose edges (edges with only one end at a vertex and the other end ``loose'') to $\Gamma$ as follows:
\begin{itemize}
\item at each vertex $v$ of $\Gamma$ attach $|s_v|$ loose edges,
\item to each loose edge assign the sign of $-s_v$. 
\end{itemize}
\end{definition}

If we call $\mathcal{L}\left(\Gamma^*\right)$ the set of loose edges and if we let $D$ be the diagonal matrix with entries given by the sums $s_v$, using the notation introduced after \autoref{def::plumb} we have \begin{align*}
&V\left(\Gamma^*\right)=V\left(\Gamma\right),\\
&E\left(\Gamma^*\right)=E\left(\Gamma\right)\cup \mathcal{L}\left(\Gamma^*\right),\\
&Q\left(\Gamma^*\right)=Q\left(\Gamma\right).
\end{align*}

To a modified plumbing graph $\Gamma^*$ with underlying plumbing graph $\Gamma$, one can associate a surface $F(\Gamma^*)$ and a set of vanishing cycles as follows: Assign to each vertex $v$ the closed orientable surface of genus $g_v$ and to each loose end of a loose edge a disk $D^2$ and connect these surfaces with tubes according to $\Gamma^*$ to obtain the surface $F(\Gamma^*)$ (i.e. for each edge, replace two disks, one in the interior of each surface corresponding to the ends of the edge, with $[0,1] \times S^1$). The vanishing cycles are simply the necks of the tubes (explicitly, the curves $\{1/2\} \times S^1 \subset [0,1] \times S^1$) used in the construction of $F(\Gamma^*)$ and each vanishing cycle's framing is equal to the sign $\pm 1$ of the edge of $\Gamma^*$ giving rise to that tube.

\begin{lemma}\label{plumb=lefs} Let $\Gamma$ be a plumbing graph. Then there exists an (achiral) Lefschetz fibration $\pi: L(\Gamma)\to D^2$ with the following properties:
\begin{enumerate}[label=(\roman*)]
\item the regular fiber of $\pi$ is diffeomorphic to $F(\Gamma^*)$, 
\item the vanishing cycles and their framings correspond to edges in $\Gamma^*$ and their signs, 
\item the monodromy $\mu$ is equal to the signed product of Dehn twists along the vanishing cycles.
\end{enumerate}

Furthermore, the 4--manifold $P(\Gamma)$ obtained as a plumbing of disk bundles of surfaces according to a plumbing graph $\Gamma$ and $L(\Gamma)$ constructed from the given vanishing cycle data are diffeomorphic. 
\end{lemma}

\begin{proof} To see that $L(\Gamma)$ is diffeomorphic to $P(\Gamma)$ we need to show that $L(\Gamma)$ is a regular neighborhood of a collection of surfaces of the right genus transversely- and self-intersecting according to $\Gamma$. Since all the vanishing cycles are disjoint on $F(\Gamma^*)$, we can see $L(\Gamma)$ as a Lefschetz fibration with exactly one singular fiber containing all the singularities. Since each vanishing cycle becomes a transverse intersection point in the singular fiber, with sign given by the sign of the vanishing cycle, we immediately get the correct configuration of surfaces. Since there is only one singular value, $L(\Gamma)$ is a neighborhood of that singular fiber.
\end{proof}

\autoref{plumb=lefs} can be combined with \autoref{C:LFtoTrisection} to obtain trisections and trisection diagrams for plumbing manifolds. For example, if $\Sigma$ is the closed orientable surface of genus $G>1$ and $p:E_n\to \Sigma$ is the disk bundle over $\Sigma$ with Euler number $n$, let $\pi:E_n\to D^2$ be the (achiral) Lefschetz fibration described in \autoref{plumb=lefs}. If $n\neq 0$, there is a $(|n|+G,|n|+2G-1;G,|n|)$ trisection of $E_n$ with diagram given by \autoref{fig::disk-surface}. If $n=0$, there is a $(G+2, 2G+1;G,2)$ trisection of $E_n$ with diagram given by \autoref{fig::disk-surface0}.

\begin{figure}[h]
\begin{subfigure}[b]{0.4\textwidth}
\centering
\includegraphics[width=\textwidth]{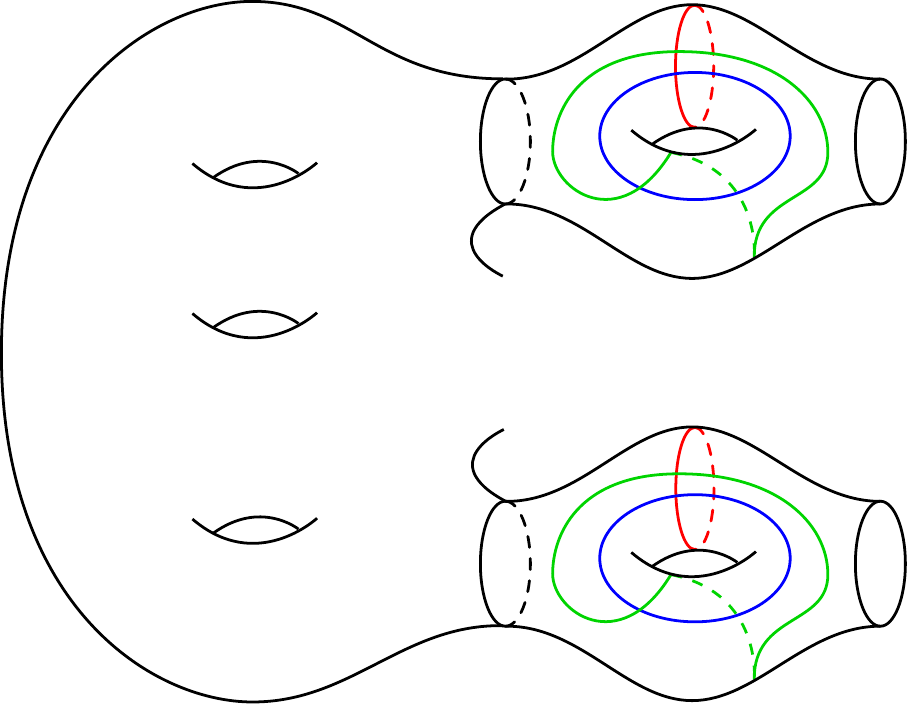}
\caption{Disk bundle over a closed surface with Euler number $n<0$. }\label{fig::disk-surface}
\end{subfigure}\hspace*{1cm}
\begin{subfigure}[b]{0.4\textwidth}
\centering
\includegraphics[width=\textwidth]{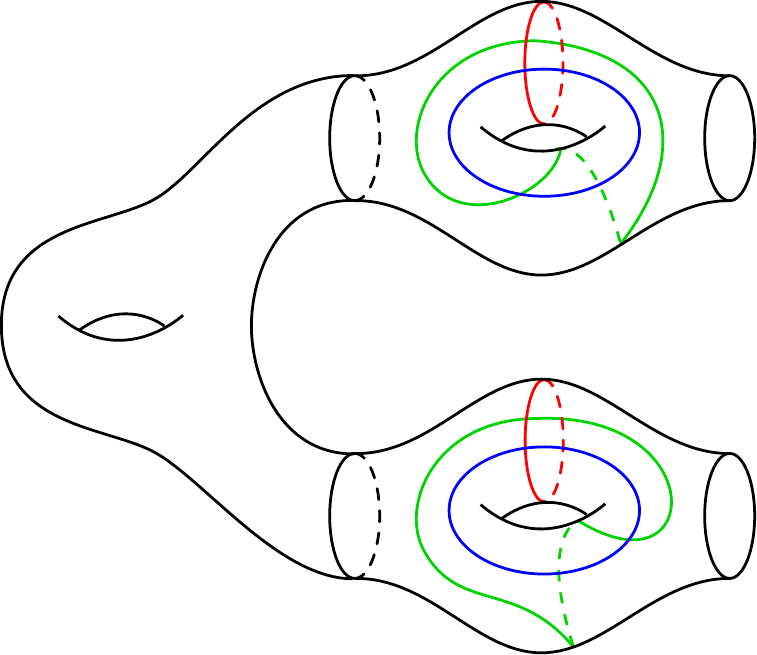}
\caption{Disk bundle over a torus with Euler number $0$.}\label{fig::disk-surface0}
\end{subfigure}
\caption{Relative trisection diagrams for the disk bundles over closed orientable surfaces.}
\end{figure}

A less trivial example is the negative definite $E_8$ manifold. The plumbing graph, the modified plumbing graph, the regular surface, and the trisection diagram are shown in \autoref{fig::E_8}.

\begin{figure}[h]
\centering
\begin{subfigure}[b]{\textwidth}
\centering
\begin{tikzpicture}[scale=0.8]
\foreach \x in {1,...,7}
    \draw[fill] (2*\x,0) circle (0.1cm);
\foreach \x in {1,...,7}
\node at (2*\x,0.5){$(0,-2)$};
\foreach \x in {2,...,7}
    \draw[-] (2*\x-2,0)--(2*\x,0) node[below,pos=0.5]{$+$};
\draw[-] (10,0)--(10,-2) node[right,pos=0.5]{$+$};
\draw[fill] (10,-2) circle (0.1cm);
\node at (10,-2.5){$(0,-2)$};
\end{tikzpicture}
\caption{The plumbing graph $E_8$.}
\end{subfigure}
\begin{subfigure}[b]{\textwidth}
\centering
\begin{tikzpicture}[scale=0.8]
\foreach \x in {1,...,7}
    \draw[fill] (2*\x,0) circle (0.1cm);
\draw[fill] (10,-2) circle (0.1cm);
\foreach \x in {2,...,7}
    \draw[-] (2*\x-2,0)--(2*\x,0) node[below,pos=0.5]{$+$};
\draw[-] (10,0)--(10,-2) node[right,pos=0.5]{$+$};
\draw[-] (2,0)--(2,1)  node[right,pos=0.5]{$+$};
\draw[-] (10,0)--(10,1)  node[right,pos=0.5]{$-$};
\draw[-] (14,0)--(14,1)  node[right,pos=0.5]{$+$};
\draw[-] (10,-2)--(10,-3) node[right,pos=0.5]{$+$};
\end{tikzpicture}
\caption{The modified plumbing graph $E_8^*$}
\end{subfigure}
\begin{subfigure}[b]{\textwidth}
\centering
\includegraphics[width=0.7\textwidth]{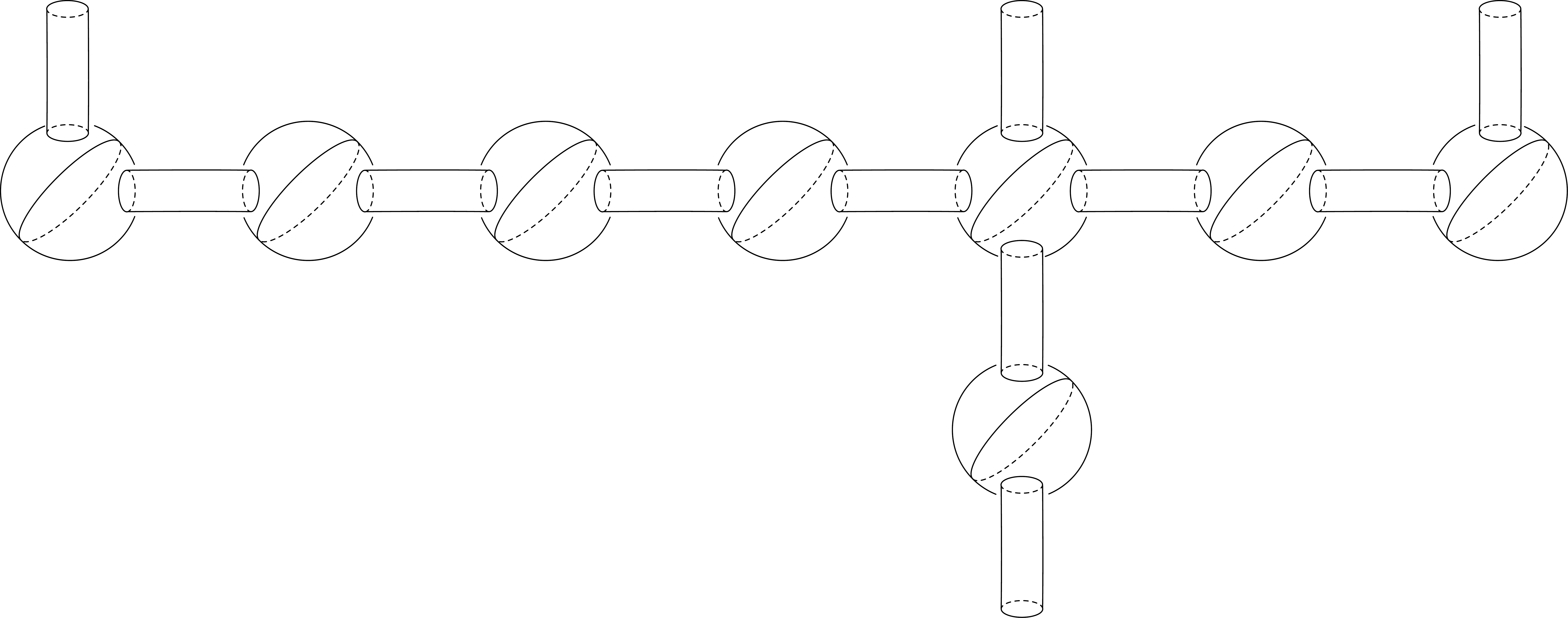}
\caption{The regular fiber of $L(E_8)$}
\end{subfigure}
\begin{subfigure}[b]{\textwidth}
\centering
\includegraphics[width=0.7\textwidth]{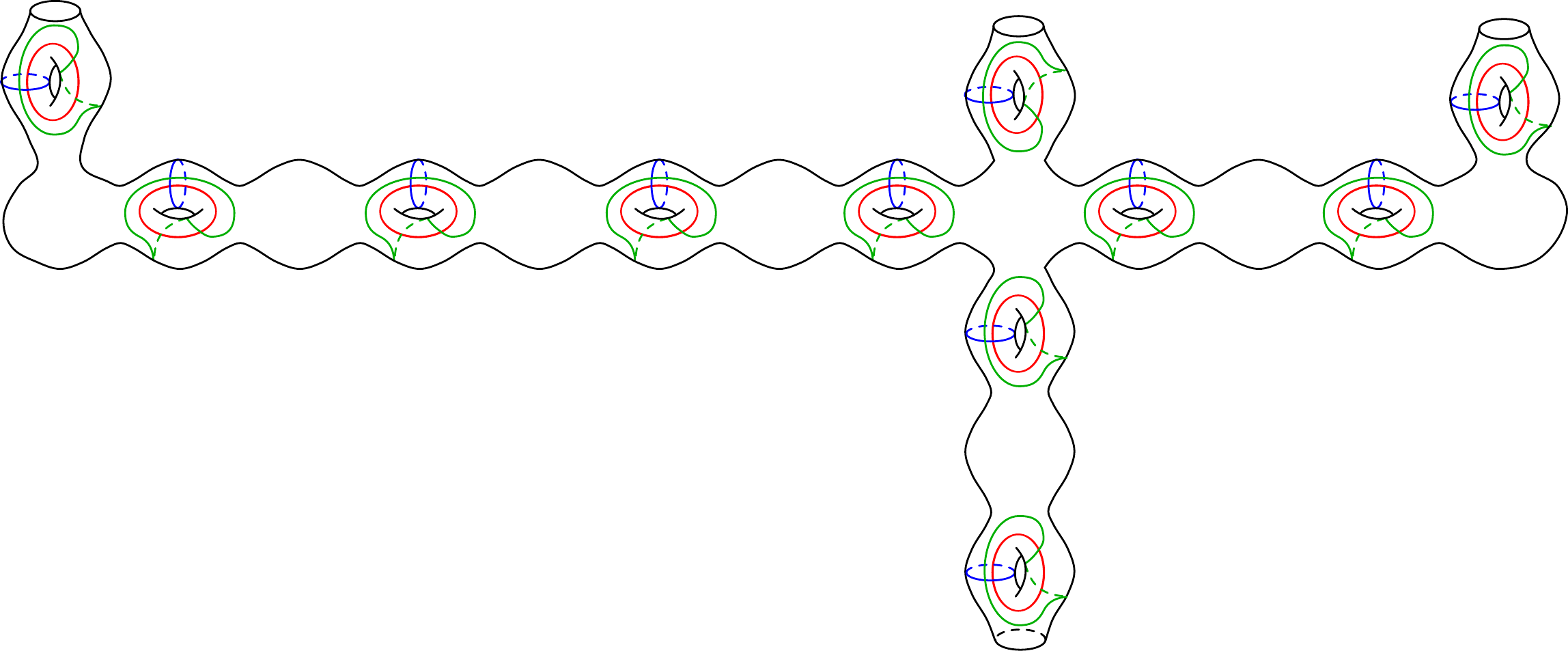}
\caption{The trisection diagram of $P(E_8)$.}
\end{subfigure}
\caption{The negative definite $E_8$ manifold. Its boundary is the Poincar\'e homology sphere.}\label{fig::E_8}
\end{figure}
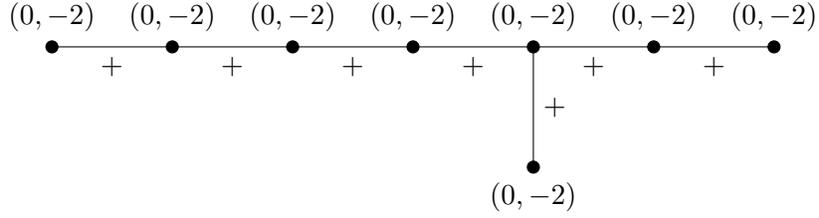
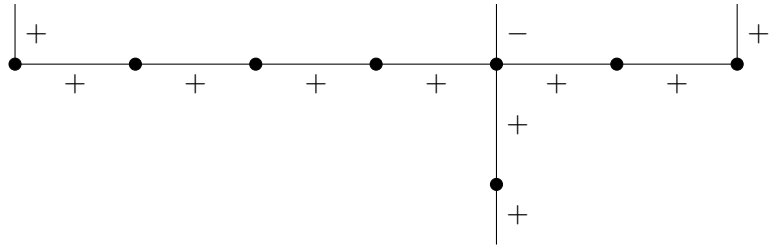
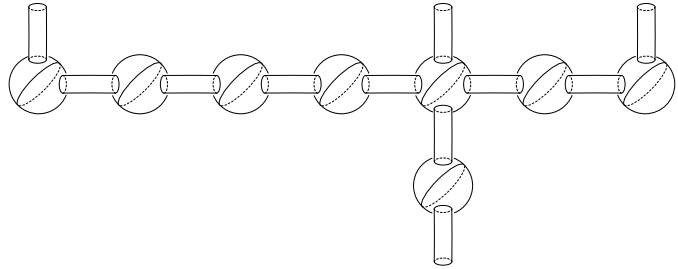
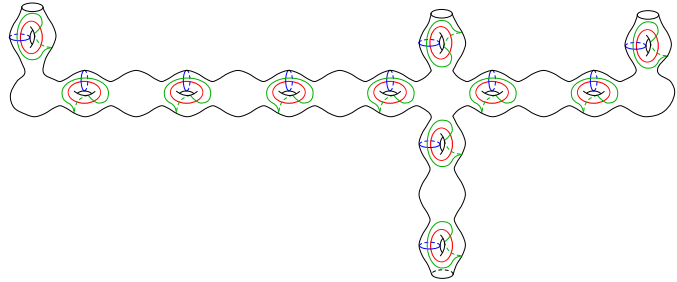

\subsection{The product of the circle with knot complements}
In this section we show that if a knot $K\subset S^3$ is in bridge position with $B$ bridges, then $X=S^1\times S^3\setminus N(K)$ admits a $(6B-1, 2B+1; 1,4)$ trisection. The description of the trisection and the trisection diagram will depend on the notion of doubly-pointed diagrams for knots in $S^3$ and so we begin the section with its definition. For the details regarding this construction we refer the reader to \cite[Section~3.2]{ras} and \cite[Example~3.4]{man-intro}.

\begin{definition} A doubly-pointed diagram for a knot $K\subset S^3$, is a tuple $(\Sigma,\mathcal{E}, \mathcal{F},z_1,z_2)$, where $(\Sigma,\mathcal{E}, \mathcal{F})$ is a Heegaard diagram for $S^3$ and $z_1$ and $z_2$ are distinct points on $\Sigma$ in the complement of $\mathcal{E}$ and $\mathcal{F}$, such that, in the associated handle decomposition of $S^3$, $K$ is the union of two arcs connecting the index $0$ and $3$ critical points, avoiding the co-cores of the $1$--handles and the cores of the $2$--handles, intersecting $\Sigma$ at $z_1$ and $z_2$.
\end{definition}

Note that if $K$ is given in bridge position with $B$ bridges, stabilizing the genus $0$ Heegaard splitting $B-1$ times gives a genus $B-1$ doubly pointed diagram describing $K$.

This description can then be translated into a Morse function $f:S^3\to [0,3]$ such that the knot $K$ is obtained as the union of the gradient flow lines of $f$ joining the unique index 3 critical point with the unique index 0 critical point and passing through the points $z_1$ and $z_2$. After a small perturbation we may assume that $f|_{\partial N(K)}$ is a standard Morse function on $T^2$; the only feature we really care about is that $f^{-1}(3/2)$ intersect $N(K)$ as meridinal disks and thus splits $\partial N(K)$ into two annuli.\\
\begin{figure}[h]
\begin{tikzpicture}[scale=1]
\draw[-] (0,0)--(6,0);
\draw[-] (0,2)--(6,2);
\draw[-] (0,4)--(6,4);
\draw[-] (0,0)--(0,4);
\draw[-] (6,0)--(6,4);
\draw[-] (1,0)--(1,2);
\draw[-] (3,0)--(3,2);
\draw[-] (5,0)--(5,2);
\draw[-] (2,2)--(2,4);
\draw[-] (4,2)--(4,4);
\node at (0.5,1){$S_2^-$};
\node at (5.5,1){$S_2^-$};
\node at (2,1){$S_1^-$};
\node at (4,1){$S_3^-$};
\node at (1,3){$S_3^+$};
\node at (3,3){$S_2^+$};
\node at (5,3){$S_1^+$};
\end{tikzpicture}
\caption{A projection of $S^1\times S^3\setminus N(K)$ into $[0,6]\times [0,3]$ using a factor of the angle of $S^1$ and the restriction of a Morse function on $S^3$ to the knot complement.}\label{fig::grid}
\end{figure}
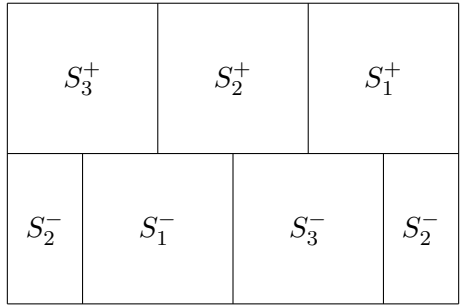

Identify $S^1$ with $[0,6]/0 \sim 6$, draw a grid on $[0,6]\times [0,3]$ as in \autoref{fig::grid} and label the squares $S_i^\pm,\; i=1,2,3$ with the sign chosen depending on the position of the square relative to the horizontal line $[0,1]\times \{3/2\}$. Notice that the left and right ends of the figure should be identified since $[0,6]$ is actually $[0,6]/0 \sim 6 = S^1$. Consider the projection $\pi: S^1\times S^3\setminus N(K)\to S^1 \times [0,3]$ given by the identity in the first component, and the restriction of the Morse function $f$ to the knot complement in the second component. Over each vertical line segment in \autoref{fig::grid} is a $3$--dimensional handlebody with $B$ $1$--handles, realized as the intersection of the genus $(B-1)$ handlebody $U^\pm$ with the knot complement $S^3\setminus N(K)$. Therefore, over each square lies a $4$--dimensional space diffeomorphic to $\natural ^B S^1\times B^3$.  Similarly, over each interior vertex lies the punctured surfaces $\Sigma'=\Sigma\setminus\left(D(z_1)\sqcup D(z_2)\right)$, where $z_1,z_2$ are the points in $\Sigma$ that describe the knot $K\subset S^3$, and $D(z_j) \; (j=1,2),$ is a disk neighborhood of $z_j$ in $\Sigma$.  We thus see that over each interior and horizontal edge of a square lies the genus $2B-1$, $3$--dimensional handlebody $I\times \Sigma'$. \\

We will obtain a trisection of $X=S^1\times S^3\setminus N(K)$ by connecting the preimage of $S_i^+$ to the preimage of $S_i^-$ using 4--dimensional 1--handles realized as tubular neighborhoods of appropriately chosen arcs. Let $k=1,\ldots,6$ and $j=1,2$, and to simplify notation identify $\partial D(z_j)$ with the unit circle in $\C$, and denote by $\xi_j^k$ the $k$-th power of a third root of unity $\xi\in S^1$, regarded as a point in $\partial D(z_j)$. Consider the arcs $a_{kj}$ obtained by taking the product of the preimage of $[k-1,k]$ in $S^1 = [0,6]/(0\sim 6)$ with the point $\xi_j^k$ in $S^3\setminus N(K)$. The $i$-th piece of the trisection of $X$ into $X_1 \cup X_2 \cup X_3$ will be obtained by connecting $\pi^{-1}(S_i^+)$ to $\pi^{-1}(S_i^-)$ using the 1--handles whose cores project into the grid as a horizontal edge disjoint from the squares $S_i^+$ and $S_i^-$, and removing from it the other 1--handles. Specifically, if we denote the tubular neighborhood of $a_{kj}$ ($k=1,\ldots,6$, $j=1,2$) in $X$ by $\nu_{kj}$, then $$X_i=\left(\pi^{-1}\left(S_i^+\sqcup S_i^-\right)\setminus \underset{\footnotesize\begin{array}{c}l\neq i,i+3\\ j=1,2\end{array}}{\cup} \nu_{lj}\right) \bigcup \left( \underset{j=1,2}{\cup} \nu_{ij}\sqcup \nu_{i+3,j} \right).$$ Since for $k\not\equiv i\mod 3$ the cores of the tubes $\nu_{kj}$ lie in the boundary of the squares $S_i^\pm$, removing them from their preimages does not change the diffeomorphism type of this space. Thus, $X_i$ is a connected space and since $\pi^{-1}(S_i^\pm)\cong \natural ^B S^1\times B^3$, we see that $X_i$ is diffeomorphic to $\natural^{2B+3} S^1\times B^3$.
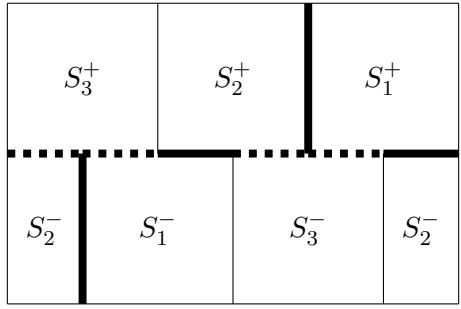
\begin{figure}[h]
\begin{tikzpicture}[scale=1]
\draw[-] (0,0)--(6,0);
\draw[-] (0,4)--(6,4);
\draw[-] (0,0)--(0,4);
\draw[-] (6,0)--(6,4);
\draw[-,line width=3pt] (1,0)--(1,2);
\draw[-] (3,0)--(3,2);
\draw[-] (5,0)--(5,2);
\draw[-] (2,2)--(2,4);
\draw[-,line width=3pt] (4,2)--(4,4);
\draw[dashed,line width=3pt] (0,2)--(1,2);
\draw[dashed,line width=3pt] (1,2)--(2,2);
\draw[-,line width=3pt] (2,2)--(3,2);
\draw[dashed,line width=3pt] (3,2)--(4,2);
\draw[dashed,line width=3pt] (4,2)--(5,2);
\draw[-,line width=3pt] (5,2)--(6,2);
\node at (0.5,1){$S_2^-$};
\node at (5.5,1){$S_2^-$};
\node at (2,1){$S_1^-$};
\node at (4,1){$S_3^-$};
\node at (1,3){$S_3^+$};
\node at (3,3){$S_2^+$};
\node at (5,3){$S_1^+$};
\end{tikzpicture}
\caption{The pieces involved in the pairwise intersection $X_1\cap X_2$.}\label{fig::grid-pair}
\end{figure}

Next we analyze the pairwise intersections of the pieces, and since the calculations are analogous for any pair $(i,i+1)$, we present the details for $X_1\cap X_2$ and leave out those concerning the other cases. There are three different types of spaces involved in the double intersection:  the pre images of the vertical segments of the intersections $S_1\cap S_2$, the preimages of the horizontal intersections, and $3$--dimensional tubular neighborhoods of some of the arcs $a_{kj}$. These sets are highlighted in \autoref{fig::grid-pair}, with the dotted line representing the presence of tubular neighborhood of two arcs. We then see that the space $X_1\cap X_2$ is diffeomorphic to the disjoint union of two 3--dimensional handlebodies of genus $B$ and two 3--dimensional handlebodies of genus $2B+1$ (two copies of $I\times \Sigma'$), connected to one another using eight 3--dimensional 1-handles. Therefore, $X_1\cap X_2$ is diffeomorphic to $\natural^{6B+3} S^1\times D^2$.

The triple intersection $F=X_1\cap X_2\cap X_3$ is the union of six copies of the punctured surface $\Sigma'=\Sigma\setminus\left(N(z_1)\sqcup N(z_2)\right)$ realized as the preimages of the six interior vertices in \autoref{fig::grid}, connected to one another using band neighborhoods of the arcs $a_{kj}$ in $S^1\times \Sigma '$. A simple computation shows that a surface so decomposed has Euler characteristic equal to $-12B$ and so, to establish the diffeomorphism type of this central surface $F$, it is enough to calculate the number of boundary components. With that in mind, notice that $\partial F$ is precisely the space $X_1\cap X_2\cap X_3\cap \partial X$, and that this space is the result of joining the copies of $\partial D(z_j)$ lying above the six internal vertices to one another using band neighborhoods of the six arcs $a_{jk}$ for $j=1,2$. For each $j=1,2$ this results in two circles, for a total of four boundary components. A schematic picture that describes these components can be found in \autoref{fig::grid-binding}. A simple Euler characteristic argument then shows that the surface $X_1\cap X_2\cap X_3$ has genus $6B-1$.

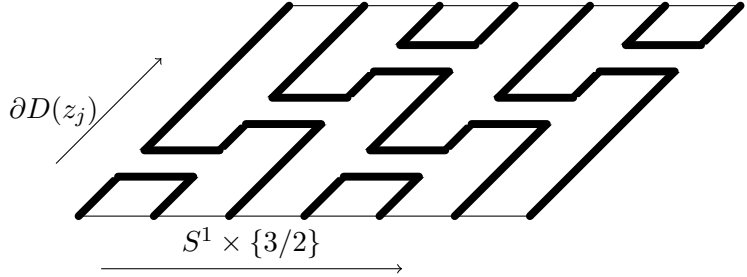
\begin{figure}[h]
\begin{tikzpicture}[scale=1,x={(1,0)}, y={(0.7,0.7)}]
\draw[-] (0,0)--(6,0);
\draw[-] (0,4)--(6,4);
\draw[->] (1,-1)--(5,-1) node[above,pos=0.5]{$S^1\times \{3/2\}$};
\draw[->] (-1,1)--(-1,3) node[left,pos=0.5]{$\partial D(z_j)$};
\foreach \x in {0,...,6}
	\draw[line width=3pt,line cap=round] (\x,0)--(\x,4);
\foreach \x in {0,1,2}
	\foreach \a in {0,3}
		\draw[fill=white,draw=white] (\a+\x-0.25,\x+0.75)--(\a+\x+1+0.25,\x+0.75)-- (\a+\x+1+0.25,\x+1.25)--(\a+\x-0.25,\x+1.25)--(\a+\x-0.25,\x+0.75);
\foreach \x in {0,1,2}
	\foreach \a in {0,3}
		\foreach \e in{-1,1}
			\draw[-,line width=3pt,line cap=round] (\a+\x,\x+1+\e/4)--(\a+\x+1,\x+1+\e/4);
\end{tikzpicture}
\caption{Two of the four components of $X_1\cap X_2\cap X_3\cap\partial X$. This parallelogram represents a torus as follows: the horizontal component represents the $S^1$ direction in the middle of \autoref{fig::grid} and the slanted direction represents the direction of $\partial D(z_j)$ which is ``internal'' to $\Sigma'$ and therefore not represented in \autoref{fig::grid}.}\label{fig::grid-binding}
\end{figure}

Next, to understand $X_1\cap X_2\cap \partial X$ intersect the highlighted pieces in \autoref{fig::grid-pair} with $\partial N(K)$. Above the vertical edges lies a cylinder, above each horizontal edge two disks realized as $I\times (\partial D(z_1)\setminus N(\xi_1^3))$ and $I\times (\partial D(z_2)\setminus N(\xi_2^3))$, and above each dotted line two band neighborhoods of the arcs (one for each of $z_1$ and $z_2$). Thus, $X_1\cap X_2\cap \partial X$ is diffeomorphic to the disjoint union of six cylinders connected to one another using eight bands. A surface with this decomposition has Euler characteristic equal to -4, and since its boundary is the same as the boundary of the central surface $F$ we conclude that $X_1\cap X_2\cap \partial X$ is a surface of genus $1$ and $4$ boundary components.

The last intersection to consider is $X_1\cap \partial X$. This space consists of two solid tori, one above each one of $S_1^\pm$, and two 3--dimensional 1--handles that lie above $[0,1]\times\{3/2\}$ and  $[3,4]\times\{3/2\}$. This shows that $X_1\cap \partial X$ is a genus 5 handlebody. Moreover, notice that each solid torus is a relative compression body from one of the cylinders in $X_1\cap X_2\cap \partial X$ to a cylinder in $X_1\cap X_3\cap \partial X$, and that each solid torus contains one of the disks in each one of $X_1\cap X_2\cap \partial X$ and $X_1\cap X_3\cap \partial X$. In addition, the 3--dimensional 1--handles are relative compression bodies between the band neighborhoods of the arcs, and so  $X_1\cap \partial X$ is diffeomorphic to the product of an interval and the surface $X_1\cap X_2\cap \partial X$.\\

Finally, to obtain a trisection diagram for $S^1\times S^3\setminus N(K)$ all that is left to do is understand the collection of disks in the pairwise intersections $X_i\cap X_{i+1}$ that are bounded by curves that lie entirely in the triple intersection $F=X_1\cap X_2\cap X_3$. One more time we focus only on the intersection $X_1\cap X_2$. In this case we have:
\begin{itemize}
\item the collection $\mathcal{F}$ of $B-1$ curves that bound disks $D_i^+$ at $\{4\}\times U^+$,
\item the collection $\mathcal{E}$ of $B-1$ curves that bound disks $D_i^-$ at $\{1\}\times U^-$,
\item a collection of $2B$ curves stemming from a handle decomposition of $[2,3]\times \Sigma'$ relative to the union of $\{2,3\}\times \Sigma'$ with band neighborhoods of the arcs $[2,3]\times\{\xi^2_j\}$, $j=1,2$. The curves are realized as the union of arcs in $\{2\}\times\Sigma'$ with arcs in $\{3\}\times\Sigma'$ going through the bands; $2(B-1)$ of the arcs arise from some 1--handles in $\Sigma'$ that give rise to genus, one other from a 1--handle in $\Sigma'$ that gives rise to the boundary components, and one other that connects the two bands. 
\item a collection analogous to the one above but related to $[5,6]\times \Sigma'$.
\end{itemize}

\begin{figure}[h!]
\centering
\begin{subfigure}[b]{\textwidth}
\centering
\begin{tikzpicture}[scale=0.8,rotate=30]
\draw (0,0) circle [radius=2cm];
\draw[fill=white] (60: 2cm) circle[radius=0.4cm] node{$\color{red}\mathcal{F}$};
\draw[fill=white] (120: 2cm) circle[radius=0.4cm] node{$\color{blue}\mathcal{E}$};
\draw[fill=white] (180: 2cm) circle[radius=0.4cm] node{$\color{Green}\mathcal{F}$};
\draw[fill=white] (240: 2cm) circle[radius=0.4cm] node{$\color{red}\mathcal{E}$};
\draw[fill=white] (300: 2cm) circle[radius=0.4cm] node{$\color{blue}\mathcal{F}$};
\draw[fill=white] (360: 2cm) circle[radius=0.4cm] node{$\color{Green}\mathcal{E}$};
\draw[fill=white,draw=white] (90: 1.5cm) circle[radius=0.4cm] node{$\color{Green}\mathcal{A}$};
\draw[fill=white,draw=white] (150: 1.5cm) circle[radius=0.4cm] node{$\color{red}\mathcal{A}$};
\draw[fill=white,draw=white] (210: 1.5cm) circle[radius=0.4cm] node{$\color{blue}\mathcal{A}$};
\draw[fill=white,draw=white] (270: 1.5cm) circle[radius=0.4cm] node{$\color{Green}\mathcal{A}$};
\draw[fill=white,draw=white] (330: 1.5cm) circle[radius=0.4cm] node{$\color{red}\mathcal{A}$};
\draw[fill=white,draw=white] (30: 1.5cm) circle[radius=0.4cm] node{$\color{blue}\mathcal{A}$};
\end{tikzpicture}
\caption{Recipe for drawing a trisection diagram for $S^1\times S^3\setminus N(K)$. The smaller circles with a letter inside represent copies of $\Sigma'$, the sub arcs of the larger circle represent the bands connecting the different copies of $\Sigma'$. Here $\mathcal{A}$ denotes the curves obtained as union of arcs arising from $I\times \Sigma'$, whereas $\mathcal{F},\mathcal{E}$ denote the curves in the doubly pointed diagram for the knot $K$. Additionally, each color represents one collection of $\alpha,\beta,\gamma$.}\label{fig::knot_recipe}
\end{subfigure}
\bigskip
\begin{tabular}{cc}
\begin{subfigure}[b]{0.3\textwidth}
\centering

\vspace*{1cm}
\includegraphics[scale=0.4]{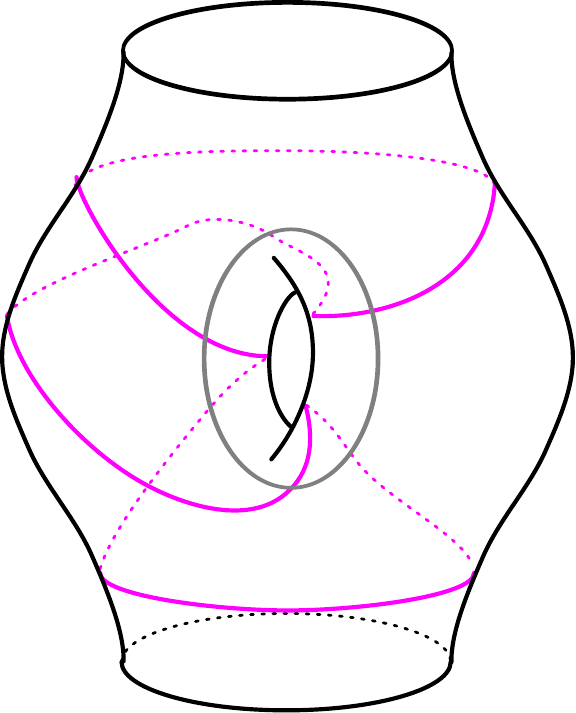}
\vspace*{1cm}
\caption{A Heegaard diagram for $S^3\setminus N(T_2,3)$. Denote by $\mathcal{F}$ the pink curve and by $\mathcal{E}$ the gray curve.}\label{fig::knot_comp-3dim}
\end{subfigure}\hspace*{\fill}&
\begin{subfigure}[b]{0.6\textwidth}
\centering
\includegraphics[scale=0.25]{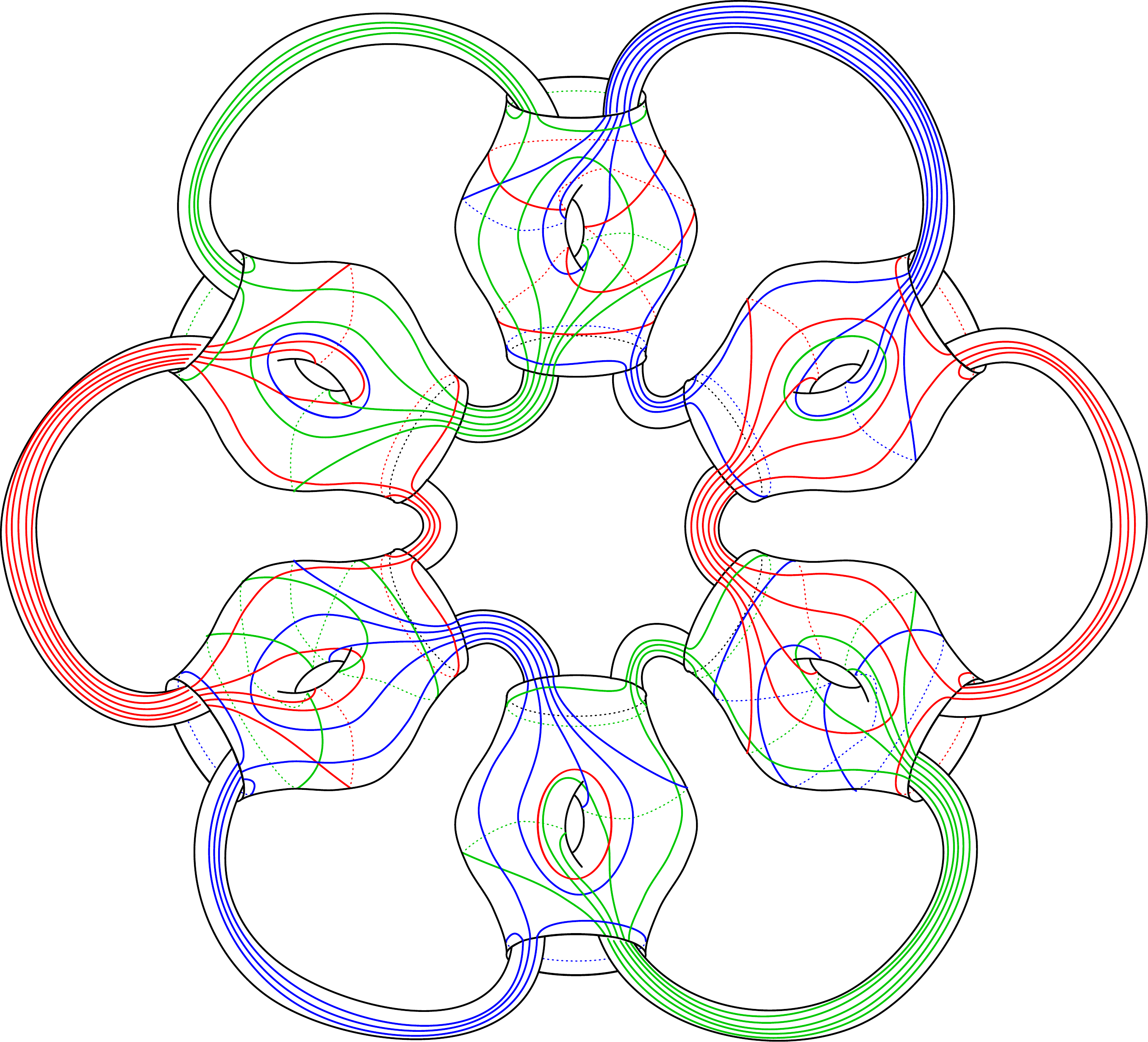}
\caption{A trisection diagram for $S^1\times S^3\setminus N(T_{2,3})$.}\label{fig::knot_comp-4dim}
\end{subfigure}
\end{tabular}
\caption{Relative trisection of $S^1\times S^3\setminus N(K)$}
\end{figure}
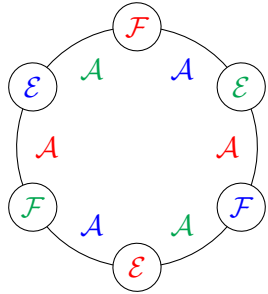
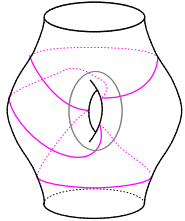
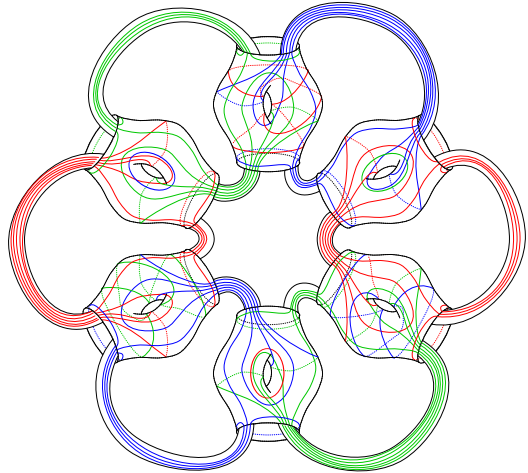

Thus, the trisection diagram consists of a surface of genus $6B-1$ with 4 boundary components, realized as the union of six copies of $\Sigma'$ joined to one another using twelve bands, and curves coming either from the Heegaard splitting of $S^3$ that corresponds to the doubly pointed diagram of $K$, or from the handlebody structure of $I\times \Sigma'$ and distributed along the pieces of $\Sigma'$ as shown in \autoref{fig::knot_recipe}.

\bibliographystyle{plain}
\bibliography{ReferencesRelTri}
\end{document}